\documentclass[
11pt,  reqno]{amsart}
\usepackage[margin=1.1in,marginparwidth=1.5cm, marginparsep=0.5cm]{geometry}

\usepackage{amsmath,amssymb,amscd,amsthm,amsxtra, esint}

\usepackage{booktabs} 
\usepackage{microtype}
\usepackage{amssymb}
\usepackage{mathrsfs}

\usepackage{upgreek} 

\usepackage{color}
\usepackage[implicit=true]{hyperref}

\usepackage{xsavebox}

\usepackage{cases}

\allowdisplaybreaks[2]

\definecolor{gr}{rgb}   {0.,   0.69,   0.23 }
\definecolor{bl}{rgb}   {0.,   0.5,   1. }
\definecolor{mg}{rgb}   {0.85,  0.,    0.85}
\definecolor{yl}{rgb}   {0.8,  0.7,   0.}
\definecolor{or}{rgb}  {0.7,0.2,0.2}

\newtheorem{theorem}{Theorem} [section]

\newtheorem{lemma}[theorem]{Lemma}
\newtheorem{proposition}[theorem]{Proposition}
\newtheorem{remark}[theorem]{Remark}

\newtheorem{definition}[theorem]{Definition}

\DeclareMathOperator*{\intt}{\int}

\newcommand{\1}{\hspace{0.5mm}\text{I}\hspace{0.5mm}}
\newcommand{\II}{\text{I \hspace{-2.8mm} I} }

\newcommand{\noi}{\noindent}
\newcommand{\Z}{\mathbb{Z}}
\newcommand{\R}{\mathbb{R}}

\newcommand{\T}{\mathbb{T}}
\newcommand{\N}{\mathbb{N}}

\let\Re=\undefined\DeclareMathOperator*{\Re}{Re}
\let\Im=\undefined\DeclareMathOperator*{\Im}{Im}

\let\P= \undefined
\newcommand{\P}{\mathbf{P}}
\newcommand{\PP}{\mathbb{P}}

\newcommand{\NN}{\mathcal{N}}
\newcommand{\RR}{\mathcal{R}}
\newcommand{\TT}{\mathcal{T}}

\newcommand{\EE}{\mathcal{E}}

\newcommand{\Nf}{\mathfrak{N}}

\newcommand{\Sf}{\mathfrak{S}}

\newcommand{\al}{\alpha}

\newcommand{\dl}{\delta}

\newcommand{\eps}{\varepsilon}

\newcommand{\G}{\Gamma}

\newcommand{\s}{\sigma}
\newcommand{\ft}{\widehat}
\newcommand{\wt}{\widetilde}
\newcommand{\cj}{\overline}
\newcommand{\dx}{\partial_x}
\newcommand{\dt}{\partial_t}

\newcommand{\embeds}{\hookrightarrow}
\newcommand{\LRA}{\Longrightarrow}

\newcommand{\Ta}{\Theta}

\newcommand{\jb}[1]
{\langle #1 \rangle}

\renewcommand{\l}{\ell}

\newcommand{\les}{\lesssim}
\newcommand{\ges}{\gtrsim}

\newcommand{\ind}{\mathbf 1}

\newcommand{\bn}{{\bf n}}

\newcommand{\JJ}{\mathcal{J}}

\numberwithin{equation}{section}
\numberwithin{theorem}{section}

\newcommand{\too}{\longrightarrow}

\newtheorem*{ackno}{Acknowledgements}

\newcommand{\Ns}{\textup{\textsf{N}}}
\newcommand{\Rs}{\textup{\textsf{R}}}

\newcommand{\M}{\mathcal{M}}

\usepackage{tikz}

\usetikzlibrary{shapes.misc}
\usetikzlibrary{shapes.symbols}
\usetikzlibrary{shapes.geometric}
\tikzset{
	dot/.style={circle,fill=black,draw=black,inner sep=1pt,minimum size=0.5mm},
	>=stealth,
	}
\tikzset{
	ddot/.style={circle,fill=white,draw=black,inner sep=2pt,minimum size=0.8mm},
	>=stealth,
	}

\tikzset{decision/.style={ 
        draw,
        diamond,
        aspect=1.5
    }}

\tikzset{dia2/.style
={diamond,fill=white,draw=black,inner sep=0pt,minimum size=1mm},
	>=stealth,
	}

\tikzset{dia/.style
={star,fill=black,draw=black,inner sep=0pt,minimum size=1mm},
	>=stealth,
	}

\makeatletter
\def\DeclareSymbol#1#2#3{\expandafter\gdef\csname MH@symb@#1\endcsname{\tikz[baseline=#2,scale=0.15]{#3}}}
\def\<#1>{\csname MH@symb@#1\endcsname}
\makeatother

\DeclareSymbol{X}{-2.4}{\node[dot] {};}
\DeclareSymbol{1}{0}{\draw[white] (-.4,0) -- (.4,0); \draw (0,0)  -- (0,1.2) node[dot] {};}
\DeclareSymbol{2}{0}{\draw (-0.5,1.2) node[dot] {} -- (0,0) -- (0.5,1.2) node[dot] {};}

\DeclareSymbol{T0}{-2.7}
 { \draw (0,0) node[ddot]{};}

\DeclareSymbol{T1}{0}
 {  
 \draw (0,0) node[dot]{} -- (0,4) node[ddot] {}; 
 \draw (-3,0) node[dot] {} -- (0,4)node[ddot] {} -- (3,0) node[dot] {};
\draw[line width=1pt] (0, 4)node[ddot, label=above:$r_1$]{} 
.. controls(3,6) and (5,6) ..
(8, 4)node[ddot, label=above:$r_2$]{}
(4, 12) node[label ] {$\underline{j = 1}$};
 }

\DeclareSymbol{T2}{0}
 {\draw (0,0) node[dot]{} -- (0,4) node[ddot] {}; 
 \draw (-3,0) node[dot] {} -- (0,4)node[ddot] {} -- (3,0) node[dot] {};
 \draw (0,-4) node[dot]{} -- (0,0) node[dot] {}; 
 \draw (-3,-4) node[dot] {} -- (0,0)node[dot] {} -- (3,-4) node[dot] {};
\draw[line width=1pt] (0, 4)node[ddot, label=above:$r_1$]{} 
.. controls(3,6) and (5,6) ..
(8, 4)node[ddot, label=above:$r_2$]{}
(4, 12) node[label ] {$\underline{j = 2}$};
 }

\DeclareSymbol{T3}{0}
 {\draw (0,0) node[dot]{} -- (0,4) node[ddot] {}; 
 \draw (-3,0) node[dot] {} -- (0,4)node[ddot] {} -- (3,0) node[dot] {};
 \draw (0,-4) node[dot]{} -- (0,0) node[dot] {}; 
 \draw (-3,-4) node[dot] {} -- (0,0)node[dot] {} -- (3,-4) node[dot] {};
\draw (10,0) node[dot]{} -- (10,4) node[ddot] {}; 
 \draw (7,0) node[dot] {} -- (10,4)node[ddot] {} -- (13,0) node[dot] {};
\draw[line width=1pt] (0, 4)node[ddot, label=above:$r_1$]{} 
.. controls(4,6) and (6,6) ..
(10, 4)node[ddot, label=above:$r_2$]{}
(5, 12) node[label ] {$\underline{j = 3}$};
 }

\tikzstyle{dot1} = [ draw=  gray!00, 
 rectangle, rounded corners, fill=gray!00,  inner sep=1pt, inner ysep=3pt]

\tikzstyle{dot2} = [ draw=  black, 
ellipse, fill=gray!00,  inner sep=1pt, inner ysep=3pt]

\tikzstyle{dot3} = [ draw=  gray!00, 
ellipse, fill=gray!00,  inner sep=1pt, inner ysep=3pt]

\DeclareSymbol{T4}{0}
 {\draw (0,0) node[dot1]{$\phantom{n_2^{(1)} = n^{(2)}}$} -- (0,15) node[dot1] {$\phantom{n^{(1)}}$}; 
 \draw (-13,0) node[dot1] {$n_1^{(1)}$} -- (0,15)node[ddot] {$\phantom{n^{(1)}}$} 
 -- (13,0) node[dot1] {$n_3^{(1)}$};
 \draw (0,-15) node[dot1]{$n_2^{(2)}$} -- (0,0) node[dot1] {$\phantom{n_2^{(1)} = n^{(2)}}$}; 
 \draw (-13,-15) node[dot1] {$n_1^{(2)}$} -- (0,0)node[dot1] {$n_2^{(1)} = n^{(2)}$} -- (13,-15) node[dot1] {$n_3^{(2)}$};
\draw (40,0) node[dot1]{{$n_2^{(3)}$}} -- (40,15) node[dot3]  {$\phantom{n^{(1)} = n^{(3)}}$} ; 
 \draw (27,0) node[dot1] {$n_1^{(3)}$} -- (40,15)node[dot3]  {$\phantom{n^{(1)} = n^{(3)}}$}  -- (53,0) node[dot1] {$n_3^{(3)}$};
\draw[line width=1pt] (0, 15)node[ddot, label=above:$r_1$]{$n^{(1)}$} 
.. controls(10,23) and (27,23) ..
(40, 15)node[dot2, label=above:$r_2$] {$n^{(1)} = n^{(3)}$} ;
 }

\begin{document}

\baselineskip = 14pt

\title[Quasi-invariant measures for the cubic 4NLS in negative Sobolev spaces]
{Quasi-invariant Gaussian measures for the cubic fourth order 
nonlinear Schr\"odinger equation 
 in negative Sobolev spaces}

\author[T.~Oh and K.~Seong]
{Tadahiro Oh and Kihoon Seong}

\address{
Tadahiro Oh\\ School of Mathematics\\
The University of Edinburgh\\
and The Maxwell Institute for the Mathematical Sciences\\
James Clerk Maxwell Building\\
The King's Buildings\\
Peter Guthrie Tait Road\\
Edinburgh\\ 
EH9 3FD\\
 United Kingdom}

\email{hiro.oh@ed.ac.uk}

\address{
Kihoon Seong\\
Department of Mathematical Sciences\\
Korea Advanced Institute of Science and Technology\\ 
291 Daehak-ro, Yuseong-gu, Daejeon 34141, Republic of Korea
}

\email{hun1022kr@kaist.ac.kr}



\begin{abstract}
We continue the study on the transport properties of the Gaussian measures
on Sobolev spaces under the dynamics of the cubic fourth order nonlinear Schr\"odinger equation.
By considering the renormalized equation, 
we extend  the quasi-invariance results  in~\cite{OTz, OST} 
to  Sobolev spaces of negative regularity.
Our proof combines 
the approach  introduced by Planchon, Tzvetkov, and Visciglia \cite{PTV}
with 
the normal form approach in \cite{OTz, OST}.
\end{abstract}


\maketitle


\tableofcontents

\newpage

\section{Introduction}
\label{SEC:1}

\subsection{Main result}  
\label{SUBSEC:4NLS}

In this paper, we study the statistical properties
of solutions
to  the cubic fourth order nonlinear Schr\"odinger equation (4NLS) on the circle
$\T=\R/(2\pi\Z )$:\footnote{The defocusing\,/\,focusing nature of the equation
does not play any role and thus we only consider the defocusing case.
The main result also applies to the focusing case.}
\begin{align}
\label{4NLS}
i\partial_tu=\partial_x^4u+  \vert u \vert^2u, 
\qquad (x,t)\in \T \times \R.
\end{align}

Let us first introduce some notations.
Given $s \in \R$, 
we consider the  Gaussian measures $\mu_s$, formally written as 
\begin{align}
d\mu_s=Z_s^{-1}e^{-\frac{1}{2}\Vert u \Vert_{H^s}^2}\,du=\prod\limits_{n\in \mathbb{Z}}Z_{s,n}^{-1}e^{-\frac{1}{2}\langle n \rangle^{2s} \vert \widehat{u}_n \vert^2} \,d\widehat{u}_n.
\label{gauss1}
\end{align}

\noi
Namely, $\mu_s$ is 
 the induced probability measure under the random Fourier series:\footnote{In the following, we often drop the harmless factor of $2\pi$.}
\begin{align}
\omega \in \Omega \longmapsto u^{\omega}(x)=u(x;\omega)=\sum_{n\in \Z}
\frac{g_n(\omega)}{\jb{n}^s}e^{inx},
\label{series1}
\end{align}

\noi
where $\jb{\,\cdot\,} =(1+| \,\cdot \,|^2)^{\frac{1}{2}}$ and $\{g_n \}_{n\in \mathbb{Z}}$ is a sequence of independent standard complex-valued Gaussian random variables\footnote{By convention, 
we set  Var$(g_n)=1$, $n \in \Z$.} on a probability space $(\Omega,\mathcal{F}, \PP)$.
It is easy to see that the random distribution \eqref{series1} belongs almost surely to $H^\s(\T)$ 
 if and only if 
\begin{align}
\s <s-\frac{1}{2}.
\label{reg1}
\end{align}

In \cite{OTz, OST}, 
with Tzvetkov and Sosoe, 
the first author studied the transport  properties
of Gaussian measures $\mu_s$ in \eqref{gauss1}
 under the 4NLS dynamics
and proved quasi-invariance\footnote{Given a measure space $(X,\mu)$, we say that $\mu$ is quasi-invariant under a measurable transformation $T:X\to X$ if the transported measure $T_{*}\mu=\mu\circ T^{-1}$ and $\mu$ are equivalent, i.e. mutually absolutely continuous with respect to each other.} of $\mu_s$, $s > \frac 12$.
Our main goal in this paper is to extend the quasi-invariance results
in \cite{OTz, OST}
to Gaussian measures on periodic distributions
of negative regularity.

It is known \cite{OTz} that 
the cubic 4NLS \eqref{4NLS} is globally well-posed in $L^2(\T)$.
Moreover, this well-posedness result is sharp in the sense that 
\eqref{4NLS} is known to be ill-posed in negative Sobolev spaces
\cite{GO, OW1}. 
Thus, in view of \eqref{reg1}, 
the quasi-invariance result for $s > \frac 12$ is optimal
since for $s \leq \frac 12$, 
the cubic 4NLS \eqref{4NLS} is almost surely ill-posed
with respect to the initial data given by the random Fourier series \eqref{series1}.
In order to study the dynamical problem in negative Sobolev spaces, 
we consider the following renormalized 4NLS:
\begin{align}
\textstyle
i\dt u=\dx^4u +\big(| u |^2-2\fint_{\T} | u |^2\,dx \big)u, 
\label{4NLS2}
\end{align}

\noi
where $\fint f(x) dx = \frac{1}{2\pi}\int f(x) dx$.
For smooth functions, the equation \eqref{4NLS2}
is equivalent to \eqref{4NLS} via the following invertible gauge transform:
\begin{equation*}
\mathcal{G}(u)(t) : = e^{ 2 i t \fint |u(t)|^2 dx} u(t).
\end{equation*}

\noi
Namely, $u \in C(\R; L^2(\T))$ satisfies \eqref{4NLS}  
if and only if $\mathcal{G}(u)$ satisfies \eqref{4NLS2}.
On the other hand, the gauge transform $\mathcal{G}$ does not make sense outside $L^2(\T)$
and thus these equations 
 describe genuinely different dynamics, if any, outside $L^2(\T)$.
As mentioned above, the original equation~\eqref{4NLS}
is ill-posed in negative Sobolev spaces.
As for  the renormalized cubic 4NLS \eqref{4NLS2}, 
the first author and Y.~Wang \cite{OW1} proved 
its  global well-posedness in $H^s(\T)$ for $s > - \frac 13$.
See also~\cite{Kwak} for local well-posedness of \eqref{4NLS2}
for $s = -\frac 13$.
See \cite{BO96, Christ, GH, OS, OW2}
for an analogous renormalization in the context 
of the usual nonlinear Schr\"odinger equation (NLS)
with the second order dispersion.
Before proceeding further, 
we point out that 
the solution map to~\eqref{4NLS2}, constructed in \cite{OW1, Kwak}, 
is not locally uniformly continuous in negative Sobolev spaces
 \cite{CO2012, OTz}. Namely, we can not construct solutions
 by a contraction argument.
This point 
will be important in our study; see Proposition~\ref{PROP:app2} below.

We now state our main result.

\begin{theorem}
\label{THM:1}
Let $s > \frac{3}{10}$. Then, the Gaussian measure $\mu_s$ in \eqref{gauss1} is quasi-invariant under the dynamics of the renormalized cubic 4NLS \eqref{4NLS2}.	
\end{theorem}

The transport properties of Gaussian measures 
have been studied extensively in probability theory;
see,  for example,  \cite{CM, Ramer, Cru1, Cru2}.
In \cite{Tz}, Tzvetkov initiated the study of
transport properties of Gaussian measures on functions\,/\,distributions under nonlinear Hamiltonian PDEs
and there has been a significant progress in this direction
\cite{Tz, OTz,  OTz2, OST, OTT,  PTV, GOTW, FT2019, STX, DT2020}.
In particular, Theorem \ref{THM:1} 
 extends the quasi-invariance results in \cite{OTz, OST}\footnote{The quasi-invariance results
 in \cite{OTz, OST} were proved for \eqref{4NLS} but they equally apply to the renormalized 4NLS~\eqref{4NLS2}.}
 to  negative Sobolev spaces  $H^\s(\T)$,  $\s > -\frac 15$.

The general strategy, as introduced in  \cite{Tz}, is to study quasi-invariance of the Gaussian measures $\mu_s$ indirectly by studying weighted Gaussian measures $\rho_s$, where the weight 
corresponds to correction terms that arise due to the presence of the nonlinearity. 
The two key steps  in this strategy are (i) the construction of the weighted Gaussian measure
$\rho_s$
and  (ii) an energy estimate on the time derivative of the modified energy
(that is,  the energy of the Gaussian measure plus the correction terms). 
It is crucial to choose 
good correction terms in order to establish an effective energy estimate.
In the context of 4NLS \eqref{4NLS}, 
this general strategy was applied
in~\cite{OTz, OST}.  In~\cite{OTz}, 
the correction term was obtained by applying a normal form reduction
(i.e.~integration by parts in time) in the spirit of 
\cite{TT, NTT, BIT, MPS}.
In the second work \cite{OST}, 
Sosoe, Tzvetkov, and the first author 
employed  an infinite iteration of normal form reductions, introduced
in \cite{GKO}, 
to compute an infinite series of correction terms
to the $H^s$-energy functional. 
Such an infinite iteration of normal form reductions has turned out to be
a useful tool in constructing solutions to PDEs
and establishing energy estimates;
see \cite{GKO, OW1,KOY,  OW2, K2019, FO}.

In order to prove Theorem \ref{THM:1}
for the renormalized 4NLS \eqref{4NLS2}
in negative Sobolev spaces, 
we also apply an infinite iteration of normal form reductions
to the $H^s$-energy functional 
and introduce
 infinitely many correction terms.
In \cite{OST}, 
the multilinear forms appearing in normal form reductions
were shown to be bounded in $L^2(\T)$.
The main task here is to extend
the  boundedness of these multilinear forms 
to negative Sobolev spaces $H^\s(\T)$, $-\frac 15 < \s < 0$.
See also
Remark~\ref{REM:NF}.
This gives rise to 
 the modified energies $\EE_{N}(u)$ 
 in~\eqref{E2}
 whose time derivatives are uniformly controlled
 on bounded sets in the support of the Gaussian measure $\mu_s$
 (see Proposition~\ref{PROP:energy}), 
 provided that $s>\frac{3}{10}$.
We point out that, as in the previous works \cite{OTz, OST}, 
 the regularity restriction in Theorem \ref{THM:1} 
 comes from the energy estimate.

The next step is to construct weighted Gaussian measures.
In  \cite{OTz, OST}, the weighted Gaussian measures
were normalized to be probability measures thanks to the (conserved) $L^2$-cutoff.
For our current problem in negative Sobolev spaces, 
however, an $L^2$-cutoff is not available
and thus 
 the weighted Gaussian measures 
 associated with the modified energies $\EE_{N}(v)$
  are not probability measures.
An important observation is that 
our proof of quasi-invariance is entirely local
in $H^{s-\frac{1}{2}-\eps}(\T)$
(see Subsection~\ref{SUBSEC:3.6}).
This allows us to work with 
the weighted Gaussian measures 
restricted to  
 compact sets in $H^{s-\frac{1}{2}-\eps}(\T)$, 
 for which we prove strong convergence
 (see Proposition~\ref{PROP:ma}).
We then follow the approach 
introduced by 
 Planchon, Tzvetkov,  and Visciglia~\cite{PTV}, 
 where they established local-in-time (and also local in the phase space) quasi-invariance properties, 
 and close the argument.

Since our argument is based on the study of frequency-truncated dynamics
(see \eqref{4NLS4}), 
an  approximation property of the truncated dynamics (Proposition \ref{PROP:app2}) 
also plays a key role.
In $L^2(\T)$, a standard contraction argument yields local well-posedness of \eqref{4NLS2}. 
By a slight variation of this contraction argument, 
 one can easily prove the desired approximation properties of the truncated dynamics 
 in $L^2(\T)$ (see \cite[Appendix B]{OTz}).
In negative Sobolev spaces, however, 
we can not use a  contraction argument to establish local well-posedness of
\eqref{4NLS2} due to the failure of local uniform continuity
of the solution map \cite{CO2012, OTz}. 
Hence, a more careful argument is required
in  studying approximation properties of the truncated dynamics. 
In fact, in negative Sobolev spaces, 
we only prove a weaker  
 approximation property of the truncated dynamics. See Remark~\ref{REM:app}.    
In Section \ref{SEC:app}, 
we discuss in detail the approximation property 
of the truncated dynamics in negative Sobolev spaces.

\begin{remark}\rm
In \cite{PTV}, the authors compared their approach
and the normal form approach in~\cite{OTz, OST}
and stated ``It would be interesting to find situations where the approaches of 
[\,\cite{OTz, OST}\,] and the one used in [\,\cite{PTV}\,]  can collaborate."
Our proof of Theorem \ref{THM:1} provides
the first such example, combining the methods from \cite{PTV} and \cite{OTz, OST}.
\end{remark}

\begin{remark} \rm
In \cite{OTzW}, Tzvetkov, Y.~Wang, and the first author
constructed global-in-time dynamics 
 for \eqref{4NLS2} almost surely with respect to the white noise, i.e.~the Gaussian measure $\mu_s$ with $s=0$.
  They also proved invariance of the white noise $\mu_0$ under \eqref{4NLS2}, which in particular implies its  quasi-invariance.
Thus, it is an interesting question to fill in the gap $0 < s \leq \frac{3}{10}$
between Theorem \ref{THM:1} and the result in \cite{OTzW}.
  
\end{remark}

\begin{remark}\rm
In \cite{LSZ}, 
  the second author with G.~Li and Zine recently proved global well-posedness of 
the following renormalized fractional NLS on $\T$ (for $\al > 2$):
\begin{align}
\textstyle
i\dt u=(-\dx^2)^\frac{\al}{2} u +\big(| u |^2-2\fint_{\T} | u |^2\,dx \big)u
\label{4NLS3}
\end{align}

\noi
in   $H^\s(\T)$ for $\s > \frac {2-\al}{6}$.
While we only consider the renormalized 4NLS \eqref{4NLS2}
 in this paper for simplicity of presentation, 
  our argument can be easily adapted to 
 study
the quasi-invariance property of $\mu_s$ under the dynamics of 
\eqref{4NLS3}
for some range of $s \le \frac 12$.

 \end{remark}

\begin{remark} \rm 

At each step  of  normal form reductions,  
we introduce a correction term.
This is precisely how correction terms are introduced in 
 the $I$-method \cite{CKSTT2003}.
In order to prove the  energy estimate (Proposition \ref{PROP:energy}), 
we implement an  infinite iteration of normal form reductions
and thus introduce an infinite series of correction terms.
In other words, the modified energies
$\EE_{N}(v)$ defined in \eqref{E2}
can be viewed as  modified energies of an infinite order in the $I$-method terminology.
Finally, we remark that
this infinite iteration of normal form reductions
allows us to encode multilinear dispersion
in the structure of the modified energy 
and thus to 
exchange analytical difficulty
with algebraic\,/\,combinatorial  difficulty.  
 
 \end{remark}

\subsection{Organization} 
In Section \ref{SEC:not}, we introduce some notations. 
In Section \ref{SEC:pf}, by assuming 
the approximation property of the truncated dynamics (Proposition \ref{PROP:app2}) and 
the energy estimate (Proposition \ref{PROP:energy})
with the related normal form reductions, we prove Theorem~\ref{THM:1}. 
In Section~\ref{SEC:app}, we discuss the approximation property of 
the truncated
dynamics.  
In Section \ref{SEC:NF}, we then establish the energy estimate (Proposition \ref{PROP:energy}) by implementing  an infinite iteration of normal form reductions.

\section{Notations}\label{SEC:not}
In the following, we fix small $\eps>0$  and set
\begin{align}
\s = s-\frac{1}{2}-\eps
\label{sig}
\end{align}

\noi
such that \eqref{reg1} is satisfied.
Given $R > 0$, we use $B_R$ to denote the ball of radius $R$ in $H^\s(\T)$
centered at the origin.

Given $N \in  \N \cup\{\infty\}$, we use $\P_{\leq N}$ to denote the Dirichlet projection onto the frequencies 
$\{|n| \leq N \}$ and set $\P_{>N}:=\text{Id}-\P_{\leq N}$. When $N=\infty$, 
it is understood that $\P_{\leq N}=\text{Id}$. Define $E_N$  by
\begin{align*}
E_N=\P_{\leq N}H^{\s}(\T)=\text{span} \{e^{inx}:\vert n \vert \leq N\}
\end{align*}

\noi
and let $E_N^\perp$ be the orthogonal complement of $E_N$
in $H^\s(\T)$.

Given $s\in \R$, let $\mu_s$ be the Gaussian measure on $H^{s-\frac{1}{2}-\eps }(\T)$ defined in \eqref{gauss1}. Then, we can write $\mu_s$ as
\begin{align}
\mu_s=\mu_{s,N}\otimes \mu_{s,N}^\perp,
\label{mu2}
\end{align}

\noi
where $\mu_{s,N}$ and $\mu_{s,N}^\perp$ are the marginal distributions of $\mu_s$ restricted onto $E_N$ and $E_N^\perp$, respectively. In other words, $\mu_{s,N}$ and $\mu_{s,N}^\perp$ are induced probability measures under the following random Fourier series:
\begin{align*}
\P_{\le N} u&:\omega \in \Omega \longmapsto 
\P_{\le N} u(x;\omega)=\sum\limits_{ \vert n \vert \leq N}\frac{g_n(\omega)}{\jb{n}^s}e^{inx},\\
\P_{> N} u&:\omega \in \Omega \longmapsto 
\P_{> N} u(x;\omega)=\sum\limits_{\vert n \vert >N}\frac{g_n(\omega
)}{\jb{n}^s}e^{inx},
\end{align*}

\noi
respectively. Formally, we can write $\mu_{s,N}$ and $\mu_{s,N}^\perp$ as
\begin{align}
d\mu_{s,N}=Z_{s,N}^{-1}e^{-\frac{1}{2} \| \P_{\leq N}u \|_{H^s}^2   }du_N 
\qquad \text{and} \qquad 
d\mu_{s,N}^\perp=\widehat{Z}_{s,N}^{-1}e^{-\frac{1}{2}\| \P_{>N}u\|_{H^s }^2  }du_N^\perp, 
\label{mu3}
\end{align}

\noi
where $d u_N$ and  $d u_N^\perp$ are (formally) the products of the Lebesgue measures
on the Fourier coefficients:
\begin{align}
du_N = \prod_{|n| \le N} d \ft u(n) 
\qquad \text{and} \qquad 
du_N^\perp = \prod_{|n| >  N} d \ft u(n) .
\label{mu4}
\end{align}

\noi
Given a function $u\in H^{s-\frac{1}{2}-\eps }(\T)$, 
we may use $u_n$ to denote the Fourier coefficient $\ft u(n)$  of $u$, when there is no confusion. 
This shorthand notation is useful in Section \ref{SEC:NF}.

We use $S(t)$ to denote the linear propagator for the fourth order Schr\"odinger equation:
\begin{align*}
S(t)  = e^{- i t \dx^4 }.
\end{align*}

\noi
We denote by $\NN(u)$ the renormalized nonlinearity in \eqref{4NLS2}:
\begin{align}
\textstyle
\NN(u)  = \big(| u |^2-2\fint_{\T} | u |^2\,dx \big)u.
\label{non1}
\end{align}

\noi
We also define the phase function $\phi(\bar n)$ by 
\begin{align}
\phi(\bar n)=\phi(n_1,n_2,n_3,n )=n_1^4-n_2^4+n_3^4-n^4.
\label{phi1}
\end{align}

\noi
Then, recall from \cite{OTz} that 
\begin{align}
\phi(\bar n )=(n_1-n_2)(n_1-n)
\Big(n_1^2+n_2^2+n_3^2+n^2+2(n_1+n_3)^2\Big)
\label{phi2}
\end{align}

\noi
under $n=n_1-n_2+n_3$.
Lastly, given $n \in \Z$ and $N \in \N$, 
we  define the index sets $\G(n)$ and $\G_N(n)$  by 
\begin{align}
\G(n)=\big\{(n_1,n_2,n_3)\in \Z^3: n=n_1-n_2+n_3 
\text{ and } n_1,n_3\neq n \big\}
\label{G1}
\end{align}

\noi
and
\begin{align}
\G_N(n)=\big\{(n_1,n_2,n_3)\in \Z^3: |n_j| \leq N, \ n=n_1-n_2+n_3 
\text{ and } n_1,n_3\neq n \big\}.
\label{G2}
\end{align}

\noi
Note that $\phi(\bar n) \ne 0$ on $\G(n)$ and $\G_N(n)$.

Given $T > 0$, we  use the following shorthand notation:
$C_TH^\s_x = C([0, T]; H^\s(\T)) $, etc.

In view of the time reversibility of the equation \eqref{4NLS2}, we only consider positive times in the following.

\section{Proof of the main result}\label{SEC:pf}

In this section, we go over the proof of Theorem \ref{THM:1}
by assuming (i) the approximation property of the truncated dynamics
(Proposition \ref{PROP:app2})
and (ii) the energy estimate (Proposition~\ref{PROP:energy})
and the analysis on the correction terms (Lemma \ref{LEM:E2}).
We present the proofs of 
Propositions~\ref{PROP:app2}
and~\ref{PROP:energy}
in Sections \ref{SEC:app} and \ref{SEC:NF}, respectively.
While
 we follow closely the structure of Section 3 in \cite{OST}, 
 we avoid using the interaction representation
 $v(t) = S(-t) u(t)$ in this section so that 
 the modified energies and the associated weighted Gaussian measures 
are not time-dependent.
Compare this with  \cite{OTz, OST}, 
where 
 the modified energies and the associated weighted Gaussian measures 
were time-dependent.

In the following, we fix $\frac 3{10} < s \le \frac 12$ and set $\s = s - \frac 12 - \eps$
for some small $\eps > 0$, unless otherwise stated.

\subsection{Truncated dynamics}

Given $N \in \N$, 
we consider the following truncated version of the renormalized 4NLS:
\begin{align}
\textstyle
i\dt u=\dx^4u +\P_{\le N} \NN(\P_{\le N} u) , 
\label{4NLS4}
\end{align}

\noi
where $\NN(u)$ is as in \eqref{non1}.
Note that \eqref{4NLS4} is not a finite-dimensional system of ODEs, 
when written on the Fourier side.
The higher frequency part $\P_{> N}u$ is propagated by the linear flow.

Given initial data $u_0 \in H^\s(\T)$, 
we can write $u_0 = \P_{\le N } u_0 + \P_{>N} u_0$.
Then,  the $L^2$-global well-posedness 
of the (renormalized) 4NLS \cite{OTz} yields
a global-in-time solution $u_N $ to the low frequency dynamics:
\begin{align}
\begin{cases}
i\dt u_N=\dx^4 u_N  +\P_{\le N} \NN(u_N )  \\
u_N |_{ t= 0} =  \P_{\le N} u_0, 
\end{cases}
\label{4NLS5}
\end{align}

\noi
while the high frequency dynamics with initial data 
$\P_{>N} u_0$ evolves  linearly  and hence is globally well-posed.
We denote by 
 $\Phi_N(t)$  the flow map of 
 the truncated dynamics \eqref{4NLS4} at time $t$:
$u(0) \in  H^{\s}(\T)  \to u(t) \in H^{\s}(\T) $.
We also denote by  $\Phi(t)$  the flow map 
to the renormalized 4NLS~\eqref{4NLS2}, 
constructed in \cite{OW1}.

We first record  the following uniform (in $N$) growth bound.
This estimate
 essentially follows from the growth bound on solutions to 
the renormalized cubic 4NLS \eqref{4NLS2}
in negative Sobolev spaces
\cite{OW1}.

 \begin{lemma}\label{LEM:growth}	
Let $\s > - \frac{1}{3}$.
Given any  $R>0$ and $T>0$, there exists $C(R,T)>0$ such that 
\begin{align*}
\Phi_N(t)(B_R)\subset B_{C(R,T)}
\end{align*}

\noi
for any $t\in [0,T]$ and  $N\in\N\cup\{\infty\}$, 
with the understanding that  $\Phi_\infty = \Phi$.
Here,  $B_R$ denotes the ball of radius $R$ in $H^\s(\T)$
centered at the origin.

\end{lemma}

Next, we state
 the  approximation property of the truncated dynamics \eqref{4NLS4}.

\begin{proposition}
\label{PROP:app2}
Let $\s > - \frac{1}{3}$.
Given $R > 0$, 
let $A\subset B_R$ be a compact set in $H^\s(\T)$.
Given  $t\in \R$, $u_0 \in A$, and small $\dl > 0$, 
 there exists $N_0=N_0(t,R, u_0, \dl)\in \N $ such that 
\begin{align*}
\Phi(t) (u_0) \in \Phi_N(t)(A+B_\dl)
\end{align*}

\noi
for any $N\geq N_0$.
\end{proposition}

We present 
the proof of Proposition \ref{PROP:app2} in Section~\ref{SEC:app}.

\begin{remark}\label{REM:app}
\rm

(i) 
It is possible to state Proposition \ref{PROP:app2}
without referring to a compact set $A$.
In fact, there exists 
 $N_0=N_0(t, u_0, \dl)\in \N $ such that 
$\Phi(t) (u_0) \in \Phi_N(t)(u_0+B_\dl)$
for any $N \ge N_0$.
We, however, stated Proposition \ref{PROP:app2}
as above so that the statement can be easily compared
with the corresponding statement in the $L^2$-setting; 
see  \cite[Proposition B.3/6.21]{OTz}.

\smallskip

\noi
(ii)
We point out that Proposition \ref{PROP:app2}
is weaker than the 
approximation property of the truncated dynamics in $L^2(\T)$,
which played a key role in the previous works \cite{OTz, OST}.
Due to the lack of local uniform continuity 
of the solution map in negative Sobolev spaces, 
the rate of approximation $N_0$ depends on the initial data $u_0$
in Proposition \ref{PROP:app2}, 
while, in $L^2(\T)$, $N_0$ does not depend on $u_0 \in A$;
 see \cite[Proposition B.3/6.21]{OTz}.
In particular, we do not know if 
we have 
$ \Phi(t) (A) \subset  \Phi_N(t)(A+B_\dl)$
for any sufficiently large $N \gg 1$.
This is different from the situation considered in \cite{PTV},
thus requiring a careful implementation of the argument.
See Subsection \ref{SUBSEC:3.6}.

We also point out that, in \cite{OTz}, 
the continuity of the solution map from
$L^2(\T)$  to the (local-in-time) $X^{0, b}$-space
was implicitly used to 
control the  high frequency part $ \P_{>N} \Phi(t)(u_0)  $ of the solution, uniformly in $u_0$
belonging to a compact set  $A\subset L^2(\T)$;
see \cite[Lemma B.1/6.19]{OTz}.
In negative Sobolev spaces, however, 
we do not know\footnote{Recall that 
the solutions constructed in \cite{OW1} belong to the short-time $X^{\s, b}$-space, 
while those constructed in~\cite{Kwak} belong to the modified $X^{\s, b}$-space which depends on initial data; see \eqref{X8} below.
In particular, we do not know if the solution map is continuous from $H^\s(\T)$ 
into the standard $X^{\s, b}$-space if $\s < 0$.} how to obtain
such a uniform control on 
the high frequency part 
$ \P_{>N} \Phi(t)(u_0)  $
for $u_0$ belonging to a compact set  $A \subset  H^\s(\T)$.

\end{remark}

\subsection{Energy estimate.} 
In this subsection, we introduce a modified $H^s$-energy functional
and state the crucial energy estimate in negative Sobolev spaces
(Proposition \ref{PROP:energy})
whose proof is presented in 
 Section \ref{SEC:NF}.

Let $N \in \N\cup\{\infty\}$.
We say that 
 $u$ is a solution to  \eqref{4NLS5}
 if $u$ is a solution to \eqref{4NLS5} 
 when $N \in \N$ and to \eqref{4NLS2} when $N = \infty$.
 Then, by iteratively applying normal form reductions as in \cite{OST}, 
we formally\footnote{For each finite $N \in \N$, 
any solution to \eqref{4NLS5} is smooth and thus the computation leading to \eqref{E1} 
does not require any justification.  See Section \ref{SEC:NF}.}
obtain the following identity:\footnote{Hereafter, we use the following shorthand notation
for multilinear form:
$  \NN^{(j)}_{0, N}(u)
= \NN^{(j)}_{0, N}(u, \dots, u)$, etc.}  
\begin{align}
\frac{d}{dt}\bigg(\frac{1}{2} \| u(t) \|_{H^s}^2  \bigg)
=\frac{d}{dt}\bigg(\sum_{j=2}^\infty \NN_{0,N}^{(j)}(u)(t)  \bigg)
+\sum_{j=2}^\infty \NN^{(j)}_{1,N}(u)(t)
+\sum_{j=2}^\infty \RR_N^{(j)}(u)(t)
\label{E1}
\end{align}	

\noi
for any (smooth) solution $u$ to the finite-dimensional truncated dynamics \eqref{4NLS5}
(i.e.~the low frequency part of \eqref{4NLS4}).
Here, 
$\NN_{0,N}^{(j)} $ is a $2j$-linear form 
and   $\NN^{(j)}_{1,N} $
and $\RR_N^{(j)} $ are $(2j+2)$-linear forms.
This motivates us to define the following modified energy:
\begin{align}\label{E2}
\EE_{N}(u):=\frac{1}{2} \| u \|_{H^s}^2-\sum\limits_{j=2}^\infty\NN_{0,N}^{(j)}(u)(t). 
\end{align}

\noi
When $N = \infty$,  we simply denote $\EE_{\infty}(u)$ by $\EE(u)$ and 
also drop the subscript $N=\infty$ from the multilinear forms; for example, we write $\NN_0^{(j)}$ for $\NN_{0,\infty}^{(j)}$.

We now state the energy estimate.

\begin{proposition}[energy estimate]
\label{PROP:energy}	
Let $\frac{3}{10}<s\leq \frac{1}{2}$ and $\s = s - \frac 12 - \eps$ for some small $\eps > 0$.
Then, given any $R > 0$ and $T>0$, the following energy estimate holds
uniformly in $N \in \N \cup\{\infty\}$\textup{:}
\begin{align*} 
\sup_{t \in [0, T]} \bigg| \frac{d}{dt} \EE_{N}(u)(t)  \bigg| \leq C_s(R)
\end{align*}

\noi
for any solution $u \in C(\R; H^\s(\T))$ to \eqref{4NLS5}, satisfying  the growth bound\textup{:}
\begin{align}
\sup_{t\in [0,T]}\|u(t) \|_{H^{\s}}\leq R.
\label{E3b}
\end{align} 

\end{proposition}

We also record the following bound on the correction terms.
Set 
\begin{align}
 \Sf_N (u) : = \sum_{j=2}^{\infty} \NN^{(j)}_{0,N}(\P_{\leq N}u )
 \label{E3a}
\end{align}

\noi
for $N \in \N \cup\{\infty\}$.

\begin{lemma}
\label{LEM:E2}
Let $\frac{1}{6}<s\leq \frac{1}{2}$.
Then,  given any  $R>0$, 
 there exists $C_s=C_s(R) > 0$ such that 
\begin{align}
| \Sf_N(u)|
\leq C_s(R) 
\label{E4}
\end{align} 

\noi
for any $ u \in B_R\subset H^\s(\T)$
and $N \in \N \cup\{\infty\}$.
Furthermore, 
$\Sf_N(u)$ converges to $\Sf_\infty(u)$ as $N \to \infty$
for each $u \in B_R$.
\end{lemma}

In Section \ref{SEC:NF}, we present the proofs of Proposition \ref{PROP:energy}
and Lemma \ref{LEM:E2}.
The main tool is 
 an infinite iteration of normal form reductions
from \cite{OST}, where such an argument was implemented in $L^2(\T)$.
For our problem, however, we need to prove boundedness
of each multilinear term by a product of the $H^\s$-norm
of $u$ with $\s = s - \frac 12 - \eps < 0$.
For this purpose, we adapt the argument from~\cite{OW1}, 
where an infinite iteration of normal form reductions
was implemented in negative Sobolev spaces.
Indeed, the only essential difference between  our argument and that in~\cite{OW1}
is the presence of the weight $\jb{n}^{2s}$, 
coming from the $H^s$-norm squared on the left-hand side of~\eqref{E1}.

\subsection{Weighted Gaussian measures} 

As in \cite{OTz, OST}, 
we prove  quasi-invariance
of the Gaussian measure $\mu_s$
indirectly by first establishing quasi-invariance
of weighted Gaussian measures
associated with the modified energies
 $\EE(u)$ and $\EE_N(u)$ in \eqref{E2}. 
In  \cite{OTz, OST},  
the weighted Gaussian measures were normalized
to be probability measures thanks to the conserved $L^2$-cutoff. 
 Due to the unavailability of a cutoff based
 on a conservation law in negative regularity, 
 we 
 do not normalize our weighted Gaussian measures 
 (which is precisely the setting for the approach in~\cite{PTV}).
 See Subsection \ref{SUBSEC:3.6}.
 
We define the following measures:
\begin{align}
d\rho_s(u)=F_s(u)d\mu_s (u)
\qquad \text{and} 
\qquad d\rho_{s,N}(u)=F_{s,N}(u)d\mu_s(u),
\label{M1}
\end{align}

\noi
where $F_s(u)$ and $F_{s, N}(u)$ are given by 
\begin{align}
F_s(u) 
& :=\exp\bigg(-\EE(u)+\frac{1}{2}\|u \|_{H^s}^2 \bigg)
= \exp\bigg(\sum_{j=2}^{\infty} \NN^{(j)}_0(u) \bigg), 
\label{M1a}\\  
F_{s,N}(u) 
& :=\exp\bigg(-\EE_N(\P_{\le N} u)+\frac{1}{2}\|\P_{\le N} u \|_{H^s}^2 \bigg)
= \exp\bigg(\sum_{j=2}^{\infty} \NN^{(j)}_{0, N}(\P_{\le N} u) \bigg).
\label{M1b}
\end{align}

\noi
We also write $\rho_{s, \infty} = \rho_s$ and  $F_{s, \infty} (u)= F_s(u)$.

Note that the quasi-invariance property is a local property
in the sense we only need to work on compact sets in $H^\s(\T)$.
Thus, 
in proving quasi-invariance of $\rho_s$ and $\rho_{s, N}$, 
we only  
 require $F_{s,N}\in L^1_{\text{loc}}(\mu_s ) $, uniformly in $N \in \N \cup\{\infty\}$,  
 (i.e.~$F_{s,N}$ is locally integrable with a uniform (in $N$) bound on each compact set) 
 and $F_{s,N}\to F_{s} $ in $L^1_{\text{loc}}(\mu_s) $.   
  
\begin{proposition}
\label{PROP:ma} 	
Let $\frac{1}{6}<s\leq \frac{1}{2}$ and 
 $\s = s - \frac 12 - \eps$ for some small $\eps > 0$.
Given any $R>0$, 
  there exists $C=C(s,R) >0$ such that 
\begin{align}
 \rho_{s,N}(B_R) =\int_{B_R} F_{s,N}(u)d\mu_s(u)
\leq C(s,R)
\label{M3}
\end{align} 

\noi
for any $N\in \N\cup\{ \infty \} $.
Namely, 
 $F_{s,N}\in L^1_{\textup{loc}}(\mu_s) $, uniformly in $N\in \N \cup \{\infty \}$.
Moreover, we have
\begin{align}
\lim_{N\to \infty }\int_{B_R} | F_{s,N}(u)-F_s(u) |d\mu_s(u)=0.
\label{M4}
\end{align}

\end{proposition}

\begin{proof}
The bound \eqref{M3} follows from 
\eqref{E4} in  Lemma \ref{LEM:E2}.
Furthermore, 
it follows from Lemma~\ref{LEM:E2}
that 
$\Sf_N(u)$  converges to  $\Sf_\infty(u)$ 
as $N \to \infty$
for each $u \in B_R$. 
Then, from \eqref{E3a}, \eqref{M1a}, and \eqref{M1b}, 
we see that $F_{s, N}$ converges to $F_s$
almost surely with respect to $\mu_s$.
Together with the bound~\eqref{E4} in  Lemma \ref{LEM:E2}, 
the bounded convergence theorem yields \eqref{M4}.
\end{proof}

\subsection{A change-of-variable formula}

Next, we go over a global aspect of the proof of Theorem~\ref{THM:1}.
From \eqref{mu2} and \eqref{mu3}, 
we can write $\rho_{s,N}$ in \eqref{M1} as 
\begin{align}
d\rho_{s,N}=Z_{s,N}^{-1}\exp\big(-\EE_{N}(\P_{\leq N}u ) \big)\, du_{N}\otimes d\mu_{s,N}^{\perp},
\label{M2}
\end{align}

\noi
where $du_N$ is  as in \eqref{mu4}
(= the Lebesgue measure on $E_N \cong \mathbb{C}^{2N+1}$)
and  $Z_{s,N}^{-1}$ is the normalizing constant for $\mu_{s,N}$.
Proceeding as in \cite{OTz} with \eqref{M2}, we have the following
change-of-variable formula.

\begin{lemma}\label{LEM:ch}
Let $\frac 16 < s \leq \frac 12$ 
and  $\s = s - \frac 12 - \eps$ for some small $\eps > 0$.
Then, we have
\begin{align*}
\rho_{s,N}(\Phi_N(t)(A) )&=\int_{\Phi_N(t)(A) } 
e^{\sum_{j=2}^\infty \NN_{0,N}^{(j)}(\P_{\leq N}u) }d\mu_s(u)\\
&=Z_{s,N}^{-1} \int_A e^{-\EE_{N}(\P_{\leq N} \Phi_N(t)   (u)  )}  du_N\otimes  d\mu_{s,N}^{\perp}
\end{align*}
for any $t\in \R$, 
 $N \in \N$, 
and any measurable set $A\subset H^\s(\T)$. 

\end{lemma}

\subsection{Proof of Theorem \ref{THM:1}}
\label{SUBSEC:3.6}

We are now ready to present the proof of Theorem \ref{THM:1}.
We follow the argument in \cite{PTV} but due to the weaker approximation
property of the truncated dynamics in negative Sobolev spaces, 
 more care is needed to close the argument.
Fix $\frac 3{10} < s \leq \frac 12$ 
and  set $\s = s - \frac 12 - \eps$ for some small $\eps > 0$.

In the following, we only consider the positive times.
Fix $t>0$. Then,  by the inner regularity of the measure $\mu_s$, it suffices to show that
\begin{align*}
A\subset H^\s(\T)\text{ compact and }\mu_s(A)=0
\qquad \LRA \qquad \mu_s(\Phi(t)(A))=0.
\end{align*} 	

\noi
Fix a compact set $A \subset H^\s(\T)$
such that $\mu_s(A) = 0$.
From Lemma \ref{LEM:E2} with Lemma \ref{LEM:growth}, we have
\begin{align}
0<\exp\bigg(\sum\limits_{j=2}^{\infty} \NN^{(j)}_0(u)   \bigg)<\infty
\label{Th1a}
\end{align}

\noi
for all $u\in A \cup\Phi(t)(A)$.
Then, it follows from  \eqref{M1} and \eqref{M1a}
with $\mu_s(A) = 0$
that $\rho_s(A) = 0$.
Our goal is to prove 
\begin{align}
 \rho_s(\Phi(t)(A))=0.
\label{Th1b}
\end{align}

\noi
Once we prove \eqref{Th1b}, 
we then conclude from \eqref{Th1a} that $\mu_s(\Phi(t)(A)) = 0$.

Since $A$ is compact, 
we have  $A\subset B_R\subset H^\s(\T)$
for some $R>0$.
By Lemma \ref{LEM:growth}, there exists 
  $C(R)>0$ such that
\begin{align} 
\Phi_{N}(\tau)(B_{2R}) \subset B_{C(R)} 
\label{Th2}
\end{align} 

\noi
for any $\tau\in [0,t]$ and $N \in \N \cup\{\infty\}$.

Fix  a measurable set $D \subset B_{2R}$.
Then, from \eqref{M2} and Lemma  \ref{LEM:ch}, we have
\begin{align}
\begin{split}
\bigg| \frac{d}{d\tau}\rho_{s,N}(\Phi_N(\tau)(D)) \bigg|
&= \bigg|     \frac{d}{d\tau} Z_{s,N}^{-1}\int_{\Phi_N(\tau)(D) }\exp\big( -\EE_{N}(\P_{\leq N}u ) \big)
\, d u_N\otimes d\mu_{s,N}^{\perp}       \bigg| \\ 
&=\bigg| Z_{s,N}^{-1}\int_{D} \frac{d}{d\tau}\exp\big(-\EE_{N}( \P_{\leq N}\Phi_N(\tau)( u) )  \big) 
\,du_N\otimes d\mu_{s,N}^{\perp}\bigg|.
\end{split}
\label{Th3}
\end{align}

\noi
From  Proposition \ref{PROP:energy} with \eqref{Th2}, we also have 
\begin{align}
\bigg| \frac{d}{d\tau}\exp\big(-\EE_{N}(\P_{\leq N} \Phi_N(\tau)(u ) ) \big) \bigg|
\leq C'(R) \exp\big(-\EE_{N}(   \P_{\leq N}\Phi_{N}(\tau)(u ) ) \big)
\label{Th4}
\end{align} 

\noi
for any $\tau\in [0,t]$ and  $u\in D$. 
Hence, from \eqref{Th3}, \eqref{Th4}, and Lemma \ref{LEM:ch}
with \eqref{M1} and \eqref{M2}, we have
\begin{align*}
\bigg|    \frac{d}{d\tau} \rho_{s,N}(\Phi_N(\tau)(D) )\bigg|
&\leq Z_{s,N}^{-1} \, C'(R)\int_D 
\exp\big(-\EE_{N}(\P_{\leq N} \Phi_N(\tau) (u )) \big) \,du_N\otimes d\mu_{s,N}^{\perp}\\   
&=Z_{s,N}^{-1}\, C'(R)\int_{\Phi_{N}(\tau)(D)} \exp\big( -\EE_{N}(\P_{\leq N}u )    \big)\,
du_N\otimes d\mu_{s,N}^{\perp}\\
&= C'(R) \int_{\Phi_{N}(\tau)(D)} F_{s,N}(u) d\mu_{s}\\
&=C'(R)\, \rho_{s,N}(\Phi_N(\tau)(D) )
\end{align*} 

\noi
for any $\tau \in [0, t]$.
Then, 
by  Gronwall's inequality, we obtain
\begin{align}
\rho_{s,N}(\Phi_{N}(\tau)(D) )
=\int_{\Phi_{N}(\tau)(D)} F_{s,N}(u) d\mu_{s}\leq \exp(C'(R) \tau )\rho_{s,N}(D)
\label{Th5}
\end{align}

\noi
for any $\tau \in [0, t]$ and  $N\in \N$.
Note that the estimate \eqref{Th5} allows us to conclude quasi-invariance
of $\rho_{s, N}$ (and $\mu_s$) under the truncated dynamics $\Phi_N(t)$.

Next, by a limiting argument, we prove quasi-invariance of $\rho_s$ under $\Phi(t)$.
From Proposition~\ref{PROP:ma}, we have
\begin{align}
\lim_{N \to \infty}\int_{B_{C(R)}}| F_{s,N}(u)-F_{s}(u)| d\mu_{s}(u)=0,
\label{Th6} 
\end{align}

\noi
where $C(R)$ is as in \eqref{Th2}.
Thus, given small $\dl > 0$, we have
\begin{equation}  
\begin{split}
\rho_{s}(\Phi(t)(A) )&=\int_{\Phi(t)(A)}F_s(u)d\mu_s\\
&= \intt_{\Phi(t)(A) \cap\Phi_N(t)(A+B_\dl) }F_s(u)d\mu_s
+\int\limits_{\Phi(t)(A)
 \setminus \Phi_N(t)(A+B_\dl) }F_s(u)d\mu_s\\
& \leq \intt_{\Phi_N(t)(A+B_\dl) }F_s(u)d\mu_s 
+  \int\limits_{\Phi(t)(A)\setminus \Phi_N(t)(A+B_\dl) }F_s(u)d\mu_s\\
& \leq \intt_{\Phi_N(t)(A+B_\dl) }F_{s,N}(u)d\mu_s+
 \intt_{\Phi(t)(A)\setminus \Phi_N(t)(A+B_\dl) }F_s(u)d\mu_s + \dl
\end{split}
\label{Th7}
\end{equation}

\noi
for any sufficiently large $N \gg1 $.
Then, by applying \eqref{Th5} 
(with $D=A+B_\dl$ for $\dl<R$)
 to~\eqref{Th7}
 and then applying \eqref{Th6} again, we have
\begin{align}
\begin{split}
\rho_{s}(\Phi(t) (A) )
& \leq \exp(C'(R)t )\int\limits_{A+B_\dl}F_{s,N}(u)d\mu_s
+ \intt_{\Phi(t)(A)\setminus \Phi_N(t)(A+B_\dl) }F_s(u)d\mu_s
+ \dl \\
& \leq \exp(C'(R)t )\int\limits_{A+B_\dl}F_{s}(u)d\mu_s
+ \intt_{\Phi(t)(A)\setminus \Phi_N(t)(A+B_\dl) }F_s(u)d\mu_s
+ 2\dl .
\end{split}
\label{Th8}
\end{align}

By Proposition \ref{PROP:ma}
and the Lebesgue dominated convergence theorem, we have 
\begin{align}
\label{Th9}
\lim_{\dl \to 0}\int_{A+B_\dl}F_{s}(u)d\mu_s=\int_A F_s(u)d\mu_s
=\rho_{s}(A)=0.
\end{align}

\noi
Next, we consider the second term on the right-hand side of \eqref{Th8}.
Let $A_N:= \Phi(t)(A)\setminus \Phi_N(t)(A+B_\dl) $. 
Then, it follows from Proposition \ref{PROP:app2} that 
\begin{align}
\limsup A_N=\bigcap_{k = 1}^\infty \bigcup_{N = k}^\infty A_N =  \varnothing. 
\label{Th10}
\end{align}

\noi
Indeed, if \eqref{Th10} did not hold, then
there would be  at least one element $u \in \limsup A_N$, 
namely,  $u \in A_N$ for infinitely many  $N$. 
This is clearly a contradiction to 
 Proposition \ref{PROP:app2}
 since,  given any such $u $ (which in particular belongs to $\Phi(t)(A)$),  we have 
\[ u  \in \Phi_N(t)(A+B_\dl) \subset A_N^c\]

\noi
for all $N \ge N_0(t,R, u, \dl)$.
This implies that $\lim_{N \to \infty}\ind_{A_N}(u) = 0$
 for any $u \in \Phi(t) (A)$
 (and thus for any $u \in H^\s(\T)$).
Hence, by  Lemma \ref{LEM:E2}
and 
the Lebesgue dominated convergence theorem,  we have 
\begin{align}
\label{Th11}
\lim\limits_{N\to \infty } 
\int_{A_N}
F_s(u)d\mu_s=0. 
\end{align}    

\noi
Finally, putting 
\eqref{Th8}, \eqref{Th9}, and \eqref{Th11} together
and taking $\dl \to 0$, 
we conclude \eqref{Th1b}.
This completes the proof of Theorem \ref{THM:1}.

\section{On the approximation property of the truncated dynamics}\label{SEC:app}
In this section, 
we study the approximation property of the truncated dynamics \eqref{4NLS4}
and present the proof of Proposition \ref{PROP:app2}.

\subsection{Gauged 4NLS}
We first go over the basic reduction of the problem.  
Fix $\s > -\frac 13$.
Let $u \in C(\R; H^\s(\T))$ be the global solution to 
the renormalized 4NLS \eqref{4NLS2} with $u|_{t = 0} = u_0$.
The main obstruction in carrying out analysis in negative Sobolev spaces is the resonant 
part of the nonlinearity.
In order to weaken the resonant interaction, 
we introduce the following gauge transform
$\JJ = \JJ_{u_0}$
as in \cite{OTzW,LSZ}:
\begin{align}
\JJ(u)(x,t)=
\JJ_{u_0}(u)(x,t) = \sum_{n\in \Z}e^{inx-it |\ft u_0 (n)|^2}\ft u (n, t).
\label{X1}
\end{align}

\noi
This gauge transform is clearly invertible and leaves the $H^s$-norm invariant. 
 A direct computation shows that the gauged function
 $v = \JJ(u)$ 
satisfies the following gauged 4NLS:
\begin{align}
\begin{cases}
i\dt v = \dx^4v + \NN_1(v)+\NN_2(v),\\
v|_{t= 0}=u_0.
\end{cases}
\label{4NLS6}
\end{align}

\noi
Here, the first nonlinearity $\NN_1(v)$ is defined by
\begin{align}
\NN_1(v)(x,t):=\sum_{n\in \Z} e^{inx}
\sum_{\G(n)}e^{it \Ta (\bar n )}
\ft v (n_1,t)\cj {\ft v(n_2,t)}\ft v (n_3,t),
\label{X1a}
\end{align}

\noi
where $\G(n)$ is as in \eqref{G1} 
and the phase function $\Ta(\bar n ) = \Ta_{u_0}(\bar n ) $ is given by 
\begin{align}
\Ta (\bar n) :=\Ta(n_1,n_2,n_3,n)
=|\ft u _0(n_1) |^2  - |\ft u _0(n_2) |^2  + |\ft u _0(n_3) |^2  - |\ft u _0(n) |^2. 
\label{X2}
\end{align}

\noi
The second nonlinearity $\NN_2(v) $ is defined by 
\begin{align}
\NN_2(v)(x,t):=-\sum_{n\in \Z}e^{inx}
\Big( |\ft v (n,t) |^2 - |\ft u_0(n)|^2\Big) \ft v (n,t).
\label{X3}
\end{align}

\noi
In the following, we often view
$\NN_1$ as a trilinear operator and, 
with a slight abuse of notations, we write $\NN_1(v_1, v_2, v_3)$
to denote the right-hand side of \eqref{X1a}, 
where we replace the $j$th occurrence of $v$ by $v_j$, $j = 1, 2, 3$.
Given a trilinear operator $\M(v_1, v_2, v_3)$, 
we write $\M(v)$ to mean $\M(v, v, v)$.
We apply this  convention in the following.

Next, we apply the gauge transform $\JJ$ in \eqref{X1} to 
 the truncated dynamics \eqref{4NLS4}.
Let $u_N \in C(\R; H^\s(\T))$ be the global solution to 
the truncated equation \eqref{4NLS4} with the same initial data $u_N|_{t = 0} = u_0$.
Then, the gauged function $v_N = \JJ(u_N)$
satisfies the following gauged truncated 4NLS:
\begin{align}
\begin{cases}
i\dt  v_N = \dx^4 v_N + \NN_1^N(v_N)+\NN_2^N(v_N)\\
v_N|_{t= 0}=u_0, 
\end{cases}
\label{4NLS7}
\end{align}

\noi
where 
 $\NN_1^N(v_N)$ and  $\NN_2^N(v_N)$ are given by 
\begin{align}
\begin{split}
\NN_1^N(v_N)(x, t): 
\! & = \P_{\le N} \NN_1(\P_{\le N}v)(x, t)\\
& = \sum_{| n|\leq N} e^{inx}\sum_{\G_N(n)  }
e^{it\Ta(\overline{n})}
\ft v_N (n_1,t)\cj {\ft v_N(n_2,t)}\ft v_N (n_3,t), \\
\NN_2^N(v_N)(x,t):
\! & =-\sum_{|n| \leq N}e^{inx}\Big( |\ft v_N(n, t)|^2  - |\ft u_0(n)|^2\Big) 
\ft v_N(n, t)\\
&\hphantom{X|}
+\sum_{|n| >N} e^{inx} |\ft u_0(n)|^2 \ft v_N (n, t).
\end{split}
\label{X4}
\end{align}

\noi
Note that the high frequency part of the solution to the gauged truncated 4NLS \eqref{4NLS7}
is given by 
\begin{align*}
\P_{>N} v_N(t)=S_{u_0}(t)\P_{>N}u_0, 
\end{align*}

\noi
where $S_{u_0}(t)$ is the modified linear propagator defined by 
\begin{align}
S_{u_0}(t) f 
: = \sum_{n \in \Z}
e^{-it(n^4+|\ft u_0(n)|^2 )}\ft f(n) e^{inx} .
\label{X6}
\end{align}

\noi
When $N=\infty$, 
the equation \eqref{4NLS7} formally reduces to \eqref{4NLS6}
and thus  
we use the notations $v_\infty$, $\NN_1^\infty(v) $,  and $\NN_2^\infty(v)$ 
for  $v$, $\NN_1(v)$, and $\NN_2(v)$
in the following.

\subsection{Function spaces and nonlinear estimates}

We recall the definition of the basic function spaces and the key  estimates
in proving local well-posedness of the renormalized 4NLS~\eqref{4NLS2}
in negative Sobolev spaces.

We first recall 
the Fourier restriction norm method introduced by Bourgain \cite{BO1}.
Given $s,b\in \R$, we define the $X^{s,b}$-space as the completion of $\mathcal{S}(\T\times \R)$ under the following norm:
\begin{align*}
\|u \|_{X^{s,b}(\T\times \R )}=\| \jb{n}^s \jb{\tau+n^4 }^b \ft u (n,\tau) \|_{\l_n^2L^2_\tau}. 
\end{align*} 

\noi
Given a time interval $I\subset \R$, 
we define the local-in-time version $X^{s,b}(I)$  by setting
\begin{align*}
\| u\|_{X^{s,b}(I)}=\inf\big\{ \| \wt u  \|_{X^{s,b}}: \wt u |_I=u  \big\}.
\end{align*}

\noi
When $I=[0,T]$, we also set $X^{s, b}_T = X^{s, b}(I)$.
We use the same notation for the time restriction of other function spaces.
Recall that 
\begin{align}
\| u \|_{C_TH^s_x} \les \| u \|_{X^{s, b}_T}
\label{X7a}
\end{align}

\noi
for $b > \frac 12$.
Using the $X^{s, b}$-space, local well-posedness in $L^2(\T)$
of 4NLS \eqref{4NLS} (and the renormalized 4NLS \eqref{4NLS2})
follows from the $L^4$-Strichartz estimate and a contraction argument.
See \cite{OTz}.

Due to the lack of local uniform continuity of the solution map,
one can not use a contraction argument to prove local well-posedness of 
the renormalized 4NLS \eqref{4NLS2} in negative Sobolev spaces. 
In \cite{OW1}, the short-time Fourier restriction norm method and the normal form approach
were used to overcome this issue.
Following the previous works \cite{TT, NTT, MT}
on the modified KdV equation and the third order NLS, 
Kwak  \cite{Kwak}
used the  modified $X^{s, b}$-space, 
defined by the norm:
\begin{align}
\|u \|_{Y_{u_0}^{s,b}(\T\times \R )}=\| \jb{n}^s \jb{\tau+n^4 -  |\ft u_0(n)|^2}^b \ft u (n,\tau) \|_{\l_n^2L^2_\tau}
\label{X8}
\end{align} 

\noi
for $u|_{t = 0} = u_0$
and proved local well-posedness of \eqref{4NLS2} by a compactness argument.
In \cite{LSZ}, Li, Zine, and the second author 
proved local well-posedness of the  fractional NLS \eqref{4NLS3} (for $\al > 2$) below $L^2(\T)$
by studying the gauged formulation (as in \eqref{4NLS6}).
We point out that 
\begin{align}
\| \JJ_{u_0}(u) \|_{X^{s, b}} = \| u \|_{Y_{u_0}^{s, b}}
\label{YY1}
\end{align}

\noi
and thus studying the renormalized 4NLS \eqref{4NLS2} in the $Y^{s, b}_{u_0}$-spaces
is equivalent to studying the gauged renormalized 4NLS \eqref{4NLS6}
in the standard $X^{s, b}$-spaces.

Let $\Psi(t)$ and $\Psi_N(t)$ be the solution maps to \eqref{4NLS6} and \eqref{4NLS7}, respectively, 
with the understanding that $\Psi_\infty(t) = \Psi(t)$.
Then, 
as a consequence of the aforementioned well-posedness results, 
 we have the following uniform growth bound
 for $\s > -\frac 13$;
given any $R > 0$ and $T > 0$, there exists $C_0(T, R)>0$ such that
\begin{align}
\sup_{N\in \N\cup \{ \infty \} }\sup_{u_0 \in B_R } 
\|\Psi_N(t)(u_0) \|_{X^{\s,\frac{1}{2}+\eps}_T}\leq C_0(T, R)
\label{growth1}  
\end{align}

\noi
for some small $\eps > 0$, 
where 
$B_R \subset H^\s(\T)$
denotes  the ball of radius $R$
centered at the origin.
For small $T = T(R) > 0$, 
the bound \eqref{growth1} follows from 
the uniform (in $N$) local well-posedness result in 
\cite{Kwak, LSZ}.\footnote{In \cite{Kwak, LSZ}, 
only the untruncated equation \eqref{4NLS2} was considered.
In view of the uniform (in $N$) boundedness
of $\P_{\leq N}$ on the relevant function spaces, 
the local well-posedness argument in \cite{Kwak, LSZ} also applies
to the truncated equation \eqref{4NLS4}, uniformly in $N \in \N$.}
For large $T > 0$, the bound \eqref{growth1}
follows from the same bound over short time intervals
together with 
 the global-in-time control and the strong uniqueness statement in \cite{OW1}
 (which guarantees that the solutions constructed in \cite{OW1, Kwak, LSZ}
 all agree)
 and 
the subadditivity of the local-in-time $X^{\s, \frac 12 + \eps}$-norms
over disjoint time intervals
as in Lemma A.4 in \cite{BOP2}.

Next, we recall the linear estimates.  See \cite{BO1, GTV}.

\begin{lemma}\label{LEM:lin}
Let $s \in \R$ and $0 < T \leq 1$.

\smallskip

\noi
\textup{(i)} For any $b \in \R$, we have
\begin{align*}
\|  S(t) u_0 \|_{X^{s, b}_T}
\leq C_b \|u_0\|_{H^{s}}.
\end{align*}

\smallskip

\noi
\textup{(ii)}
Let $ - \frac 12 < b' \leq 0 \leq b \leq b'+1$.
Then,  we have 
\begin{align*}
\bigg\|  \int_0^t S(t-t')F(t') dt'\bigg\|_{X^{s, b}_T}
\leq C_{b, b'} T^{1-b+b'} \|F\|_{X^{s, b'}_T}.
\end{align*}

\end{lemma}

We now state the nonlinear estimates, which essentially follow from \cite{Kwak, LSZ}.  
In the remaining part of this section,  we fix small $\eps > 0$.

\begin{lemma}
\label{LEM:T1}
Let $-\frac{1}{2}<\s<0$ and $T>0$. Then,  
 we have
\begin{align}
\| \NN_1^N(v_1,v_2,v_3) \|_{X^{\s,-\frac{1}{2}+2\eps  }_T} 
& \les \prod_{j = 1}^3  \| v_j \|_{X^{\s,\frac{1}{2}+ \eps }_T}, 
\label{YY2a}
\end{align}

\noi
uniformly in  $N\in \N\cup \{\infty \}$.
\end{lemma}

\begin{proof}
This is a direct consequence of 
Proposition 3.1 in \cite{Kwak}.
Indeed, in terms of  our notations, 
Proposition~3.1 in~\cite{Kwak} establishes the following trilinear bound:
\begin{align}
\| \Ns_1(u_1,u_2,u_3) \|_{Y^{\s,-\frac{1}{2}+2\eps  }_{u_0, T}} 
& \les \prod_{j = 1}^3  \| u_j \|_{Y^{\s,\frac{1}{2}}_{u_0, T}}
\label{YY2}
\end{align}

\noi
for $-\frac{1}{2}<\s<0$ and $0 < T\leq 1$, 
where $\Ns_1$ is defined by
\begin{align}
\Ns_1(u_1, u_2, u_3)(x,t):=\sum_{n\in \Z} e^{inx}
\sum_{\G(n)}
\ft u_1 (n_1,t)\cj {\ft u_2(n_2,t)}\ft u_3 (n_3,t).
\label{YY3}
\end{align}

\noi
We first note that the restriction $T\le 1$ in \eqref{YY2} 
does not play any role in the proof presented in~\cite{Kwak} 
and thus we can drop the restriction $T \leq 1$.
A similar comment applies to the lemmas below.

From 
\eqref{X4} 
and 
\eqref{YY3}
with 
\eqref{X1} and \eqref{X2}, we have
\begin{align}
\NN_1^N(v_1,v_2,v_3)
= \P_{\leq N} \JJ \big(\Ns_1( \P_{\leq N} u_1, \P_{\leq N} 
u_2, \P_{\leq N} u_3) \big), 
\label{YY4}
\end{align}

\noi
where 
$u_j = \JJ^{-1}(v_j)$.
%
Then, 
from \eqref{YY4}, \eqref{YY1},  and \eqref{YY2} together with the uniform (in $N$)
boundedness of $\P_{\le N}$ on the $X^{s, b}_T$-
and $Y^{s, b}_{u_0, T}$-spaces, we have 
\begin{align*}
\| \NN_1^N(v_1,v_2,v_3) \|_{X^{\s,-\frac{1}{2}+2\eps  }_T} 
& = 
\| \P_{\leq N} \Ns_1(\P_{\leq N}u_1,\P_{\leq N}u_2,\P_{\leq N}u_3) \|_{Y^{\s,-\frac{1}{2}+2\eps  }_{u_0, T}} \\
& \les \prod_{j = 1}^3  \| \P_{\le N} u_j \|_{Y^{\s,\frac{1}{2} }_{u_0, T}} 
\le  \prod_{j = 1}^3  \| v_j \|_{X^{\s,\frac{1}{2}+ \eps }_T}, 
\end{align*}

\noi
where, in the last step,  we used the monotonicity of the $X^{s, b}$-norm in 
the parameter $b$.
This yields \eqref{YY2a}.
\end{proof}

\begin{lemma}
\label{LEM:T2}	
Let $-\frac{1}{3}<   \s <0 $ and $T>0$. 
Given $N \in \N \cup\{\infty\}$, 
let  $v_N$ be  the smooth solution to \eqref{4NLS7} 
 with $v_N|_{t = 0} = u_0 \in C^\infty(\T)$. 
 Then, we have 
\begin{equation}
\begin{split}
\sup_{|n|\leq N} 
& \bigg| \Im  \bigg( \int_0^T \sum_{\G_N(n)}
e^{it \Ta(\bar n)}
\ft v_N (n_1, t)\cj{\ft v_N (n_2, t)}\ft v_N (n_3, t)\cj{\ft v_N (n, t)} dt \bigg)   \bigg|\\
& \les  
\| v_N\|_{X^{\s,\frac{1}{2}+\eps}_T }^4 + 
\| v_N\|_{X^{\s,\frac{1}{2}+\eps}_T }^6+\|v_N\|_{X^{\s,\frac{1}{2}+ \eps}_T}^8,
\end{split}
\label{X10}
\end{equation}

\noi
where $\G_N(n)$ is as in \eqref{G2} and $\G_N(n) = \G(n)$ when $N = \infty$.
Here, the implicit constant in~\eqref{X10} is independent of 
 $N \in \N$.

\end{lemma}

\begin{proof}
\noi
$\bullet$ {\bf Case  1:} $N = \infty$.
\quad 
We first recall Proposition 3.4 in \cite{Kwak}; 
given  $-\frac{1}{3}\leq    \s <0 $ and $0 < T \leq 1$, 
 we have 
\begin{equation}
\begin{split}
\sup_{n \in \Z}
& \bigg| \Im  \bigg( \int_0^T \sum_{\G(n)}
\ft u  (n_1, t)\cj{\ft u  (n_2, t)}\ft u  (n_3, t)\cj{\ft u  (n, t)} dt \bigg)   \bigg|\\
& \les \|u_0 \|_{H^\s}^4 + 
\Big( \|u_0 \|_{H^\s}^2+ 
\|u\|_{Y^{\s,\frac{1}{2} }_{u_0, T}} ^4\Big)^2 + 
\|u\|_{Y^{\s,\frac{1}{2} }_{u_0, T}} ^4 + 
\|u\|_{Y^{\s,\frac{1}{2} }_{u_0, T}} ^6
\end{split}
\label{YY6}
\end{equation}

\noi
for any smooth solution  $u$ to \eqref{4NLS2}
with $u |_{t = 0} = u_0 \in C^\infty(\T)$.
As mentioned in the proof of Lemma \ref{LEM:T1}, 
we can drop the restriction $T \leq 1$ and 
the estimate \eqref{YY6} indeed holds for any $T > 0$
(at least for {\it smooth} solutions; see Remark \ref{REM:T4} below).

Given $u_0 \in C^\infty(\T)$, 
let $v$ be the solution to \eqref{4NLS6}
with $v|_{t = 0} = u_0$.
Then, 
$u = \JJ^{-1}(v)$ satisfies~\eqref{4NLS2}
with $u|_{t = 0} = u_0$.
Hence, from 
\eqref{YY6}
with 
\eqref{X1},  \eqref{X2}, 
and \eqref{YY1}, 
 we obtain
\begin{equation*}
\begin{split}
\sup_{n \in \Z} 
& \bigg| \Im  \bigg( \int_0^T \sum_{\G(n)}
e^{it \Ta(\bar n)}
\ft v (n_1, t)\cj{\ft v (n_2, t)}\ft v (n_3, t)\cj{\ft v (n, t)} dt \bigg)   \bigg|\\
& \les \|u_0 \|_{H^\s}^4 + 
\Big( \|u_0 \|_{H^\s}^2+ 
\|v\|_{X^{\s,\frac{1}{2}}_T} ^4\Big)^2 + 
 \| v\|_{X^{\s,\frac{1}{2}}_T }^4+\|v\|_{X^{\s,\frac{1}{2}}_T}^6\\
& \les 
\| v\|_{X^{\s,\frac{1}{2}+\eps}_T }^4 + 
\| v\|_{X^{\s,\frac{1}{2}+\eps}_T }^6+\|v\|_{X^{\s,\frac{1}{2}+ \eps}_T}^8,
\end{split}
\end{equation*}

\noi
where, 
by relaxing the temporal regularity from $b = \frac 12$ 
to $b = \frac 12 + \eps$, we used \eqref{X7a}
in the last step.
This proves \eqref{X10} for $N = \infty$.

\medskip

\noi
$\bullet$ {\bf Case  2:} $N < \infty$.
\quad 
As in the case $N = \infty$, 
we establish \eqref{X10} for $N < \infty$
by reducing the estimate to an analogue of~\eqref{YY6}.
For this purpose, 
we first recall  the proof of 
 Proposition 3.4 in~\cite{Kwak} (namely, the estimate \eqref{YY6}).
 First, we 
divide the domain 
 $\G(n)$ into 
a good region 
 $\G^\text{good}(n)$ and a bad region $\G^\text{bad}(n)$.\footnote{The precise definitions
 of $\G^\text{good}(n)$ and a bad region $\G^\text{bad}(n)$ are not important for our purpose.}
Then, 
the good part, i.e.~the contribution to \eqref{YY6} from  $\G^\text{good}(n)$,  is treated by 
establishing 4-linear estimates  (Cases II and III in the proof 
of \cite[Proposition 3.4]{Kwak}), 
yielding the third term on the right-hand side of \eqref{YY6}.
In handling the bad part, 
i.e.~the contribution to \eqref{YY6} from  $\G^\text{bad}(n)$
  (corresponding to Case I in the proof 
of \cite[Proposition 3.4]{Kwak}),
we first apply 
 integration by parts in time
 (as in \cite{TT, NTT, MT})
 and write
 \begin{equation}
\begin{split}
    \int_0^T  & \sum_{\G^\text{bad}(n)}
\ft u  (n_1, t)\cj{\ft u  (n_2, t)}\ft u  (n_3, t)\cj{\ft u  (n, t)} dt\\
& 
=  \int_0^T \sum_{\G^\text{bad}(n)}
e^{-i\phi(\bar n)t}\ft w  (n_1, t)\cj{\ft w  (n_2, t)}\ft w  (n_3, t)\cj{\ft w  (n, t)} dt
 \\
& 
=  \sum_{\G^\text{bad}(n)}
\frac{e^{-i\phi(\bar n)t}}{-i\phi(\bar n)}\ft w  (n_1, t)\cj{\ft w  (n_2, t)}\ft w  (n_3, t)\cj{\ft w  (n, t)}
\bigg|_{t =0}^T\\
& 
\quad 
+    \int_0^T \sum_{\G^\text{bad}(n)}
\frac{e^{-i\phi(\bar n)t}}{i\phi(\bar n)}
\dt \Big(\ft w  (n_1, t)\cj{\ft w  (n_2, t)}\ft w  (n_3, t)\cj{\ft w  (n, t)} \Big)dt\\
& =: \1_n + \II_n, 
\end{split}
\label{YY8}
\end{equation}
 
\noi
where $\phi(\bar n)$ is as in \eqref{phi1} 
and $w(t) = S(-t) u(t)$ denotes 
the interaction representation of $u$.
As for $\1_n$, 
a simple 4-linear estimate yields
\begin{equation*}
\sup_{n \in \Z} |\1_n|
 \les \|u_0 \|_{H^\s}^4 + \|u(T) \|_{H^\s}^4. 
\end{equation*}

\noi
Combining this with the following bound (see Corollary 3.3 in \cite{Kwak}):
\begin{align*}
 \|u(T) \|_{H^\s}^2  \les \|u_0 \|_{H^\s}^2 + \|u\|_{Y^{\s,\frac{1}{2} }_{u_0, T}} ^4,  
\end{align*}

\noi
we obtain
\begin{equation*}
\sup_{n \in \Z} |\1_n|
 \les \|u_0 \|_{H^\s}^4 + 
\Big( \|u_0 \|_{H^\s}^2+ 
\|u\|_{Y^{\s,\frac{1}{2} }_{u_0, T}} ^4\Big)^2, 
\end{equation*}

\noi
yielding the first two  terms on the right-hand side of \eqref{YY6}.
As for $\II_n$, recalling that $u$ satisfies~\eqref{4NLS2}, 
we see that $w(t) = S(-t) u(t)$ satisfies 
\begin{align}
 i \dt w = S(-t) \NN(S(t) u).
 \label{YY9}
\end{align}

\noi
See also \eqref{4NLS8} below.
By applying the product rule in taking a time derivative in \eqref{YY8}
and substituting \eqref{YY9},
we express $\II_n$ as 
a sum of  6-linear  terms, each of which 
can be bounded by 
establishing 6-linear estimates.
This yields the fourth term on the right-hand side of \eqref{YY6}.

In establishing \eqref{X10} for $N < \infty$, 
we repeat the argument in Case 1
 and first reduce the proof of~\eqref{X10} to 
establishing the following analogue of \eqref{YY6}:
\begin{equation}
\begin{split}
\sup_{|n| \leq N}
& \bigg| \Im  \bigg( \int_0^T \sum_{\G_N(n)}
\ft u_N  (n_1, t)\cj{\ft u_N  (n_2, t)}\ft u_N  (n_3, t)\cj{\ft u_N  (n, t)} dt \bigg)   \bigg|\\
& \les \|u_0 \|_{H^\s}^4 + 
\Big( \|u_0 \|_{H^\s}^2+ 
\|u_N\|_{Y^{\s,\frac{1}{2} }_{u_0, T}} ^4\Big)^2 + 
\|u_N\|_{Y^{\s,\frac{1}{2} }_{u_0, T}} ^4 + 
\|u_N\|_{Y^{\s,\frac{1}{2} }_{u_0, T}} ^6
\end{split}
\label{YY10}
\end{equation}

\noi
for any smooth solution  $u_N$ to \eqref{4NLS4}
with $u |_{t = 0} = u_0 \in C^\infty(\T)$.
Once \eqref{YY10} is established, 
we can simply repeat the reduction in Case 1 
(with $u_N = \JJ^{-1}(v_N)$) 
and obtain \eqref{X10}
for $N < \infty$.

Lastly, note that 
the only difference between the equations \eqref{4NLS4} and \eqref{4NLS2} is the presence
of the frequency cutoff $\P_{\le N}$.
Hence, in view of the uniform (in $N$) boundedness
of $\P_{\le N}$ on the relevant spaces, 
we see that the proof of
\eqref{YY6} described above
(namely, the proof of \cite[Proposition 3.4]{Kwak})
can be  directly applied\footnote{including the integration-by-parts
argument in \eqref{YY8}.  We just need to insert $\P_{\le N}$
in appropriate places.} to  establish \eqref{YY10}
for $N < \infty$.
This concludes the proof of 
Lemma \ref{LEM:T2}.
\end{proof}

\begin{lemma}\label{LEM:T3}
Let $-\frac{1}{3} <  \s <0 $ and $T>0$. 
Given $N \in \N$, 
let  $v$ and $v_N$  be  the smooth solutions to \eqref{4NLS6} and \eqref{4NLS7}, respectively, 
 with $v|_{t = 0} = v_N|_{t = 0} = u_0 \in C^\infty(\T)$. 
Then, we have
\begin{align}
\begin{split}
 \sup_{|n|\le N}
& \bigg|   \Im  \bigg( \int_0^T \sum_{\G_N(n)}
e^{it \Ta(\bar n)}
\Big( 
\ft v (n_1, t)\cj{\ft v (n_2, t)}\ft v (n_3, t)\cj{\ft v (n, t)}\\
& \hphantom{XXXXXXXXXX}
- 
\ft v_N (n_1, t)\cj{\ft v_N (n_2, t)}\ft v_N (n_3, t)\cj{\ft v_N (n, t)} \Big)dt  \bigg)   \bigg|\\
& \leq  C\Big(  \| v \|_{X^{\s,\frac{1}{2}+\eps}_T }, 
\|v_N \|_{X^{\s,\frac{1}{2}+\eps}_T } \Big) 
\Big( \|v - v_N\|_{X^{\s,\frac{1}{2}+\eps }_T} + \|\P_{> \frac N3} v \|_{X^{\s,\frac{1}{2}+\eps }_T} \Big), 
\end{split}
\label{X11}
\end{align}

\noi
where  the implicit constant in~\eqref{X11} is independent of 
 $N \in \N$.

\end{lemma}

\begin{proof}
This lemma follows from a slight modification of the proof of 
Proposition 3.8 in \cite{Kwak}
which states the following difference estimate; 
given  $-\frac{1}{3}\leq    \s <0 $ and $0 < T \leq 1$, 
 we have 
\begin{align}
\begin{split}
 \sup_{n \in \Z}
& \bigg|   \Im  \bigg( \int_0^T \sum_{\G(n)}
\Big( 
\ft u_1 (n_1, t)\cj{\ft u_1 (n_2, t)}\ft u_1 (n_3, t)\cj{\ft u_1 (n, t)}\\
& \hphantom{XXXXXXXXXX}
- 
\ft u_2 (n_1, t)\cj{\ft u_2 (n_2, t)}\ft u_2 (n_3, t)\cj{\ft u_2 (n, t)} \Big)dt  \bigg)   \bigg|\\
& \leq  C\Big( \|u_0\|_{H^\s},  \| u_1 \|_{Y^{\s,\frac{1}{2} }_{u_0, T}}, 
\|u_2 \|_{Y^{\s,\frac{1}{2} }_{u_0, T}} \Big) 
 \|u_1 - u_2\|_{Y^{\s,\frac{1}{2} }_{u_0, T}}
\end{split}
\label{XY1}
\end{align}

\noi
for any smooth solutions  $u_1, u_2$ to \eqref{4NLS2}
with $u_1 |_{t = 0} =u_2 |_{t = 0} = u_0 \in C^\infty(\T)$.
As before, 
we can drop the restriction $T \leq 1$ and 
the estimate \eqref{XY1} holds for any $T > 0$
(at least for {\it smooth} solutions; see Remark \ref{REM:T4} below).
The proof of \eqref{XY1} is analogous to that of \eqref{YY6}
(i.e.~Proposition 3.4 in \cite{Kwak}).
Namely, 
divide the domain 
 $\G(n)$ into 
a good region 
 $\G^\text{good}(n)$ and a bad region $\G^\text{bad}(n)$.
 Then, the good part is estimated
 by the same 4-linear estimates 
 as in the proof of \eqref{YY6}, 
 while,  as for the bad part, we apply 
 integration by parts at the level of the interaction representation
 (as in \eqref{YY8})
 and 
rewrite  the 4-linear terms 
into the 4-linear boundary terms and the 6-linear terms.

In order to prove \eqref{X11}, 
we aim to bound the following difference:
\begin{align}
\begin{split}
 \sup_{|n|\leq N}
& \bigg|   \Im  \bigg( \int_0^T \sum_{\G_N(n)}
\Big( 
\ft u (n_1, t)\cj{\ft u (n_2, t)}\ft u (n_3, t)\cj{\ft u (n, t)}\\
& \hphantom{XXXXXXXXXX}
- 
\ft u_N (n_1, t)\cj{\ft u_N (n_2, t)}\ft u_N (n_3, t)\cj{\ft u_N (n, t)} \Big)dt  \bigg)   \bigg|, 
\end{split}
\label{XY2}
\end{align}

\noi
where $u$ and $u_N$ are solutions 
to  \eqref{4NLS2} and \eqref{4NLS4}, respectively, 
with $u |_{t = 0} =u_N |_{t = 0} = u_0 \in C^\infty(\T)$.
We proceed as in the proof of \eqref{XY1} (= Proposition 3.8 in \cite{Kwak})
described above.
In studying~\eqref{XY1},  a difference appears in the integration-by-parts step
(in estimating the contribution from the bad region $\G^\text{bad}(n)$).
After applying integration by parts to 
the first summand in~\eqref{XY2}, the non-boundary  looks like 
\begin{align}
\int_0^T 
\frac{e^{i \phi(\bar n)t}}{i \phi(\bar n)}
\dt  \Big( \ft w (n_1, t)   \cj{\ft w (n_2, t)}\ft w (n_3, t)\cj{\ft w (n, t)} \Big)dt , 
\label{XY3}
\end{align}

\noi
where $\phi(\bar n)$ is as in \eqref{phi1}
and $w= S(-t) u(t)$ is the interaction representation of $u$.
See \eqref{YY8}.
We then apply the product rule and use \eqref{YY9}
to replace $\dt \ft w$ by  (the Fourier transform of) the cubic nonlinearity:
$\M(w)(t) : = S(-t) \NN(S(t) w(t)) $.
Write 
\begin{align} 
\M(w) = \P_{\le N} \M(\P_{\le N} w)
+ \P_{\le N} \big( \M(w) - \M(\P_{\le N} w)\big)
+ \P_{> N} \M( w).
\label{X13}
\end{align}

\noi
The first term on the right-hand side of \eqref{X13}
can be put together with the analogous contribution 
for $w_N(t) = S(-t) u_N(t) $ coming
from the 
 second summand in
 \eqref{XY2}, yielding
\begin{align} 
\begin{split}
\P_{\le N} & \M(\P_{\le N} w)
- \P_{\le N} \M(\P_{\le N} w_N)\\
& = \P_{\le N} \M(\P_{\le N} (w - w_N), \P_{\le N} w, \P_{\le N} w)\\
& \quad + \P_{\le N} \M(\P_{\le N}  w_N, \P_{\le N} (w - w_N) , \P_{\le N} w)\\
& \quad + \P_{\le N} \M(\P_{\le N}  w_N, \P_{\le N}  w_N , \P_{\le N} (w-w_N)).
\end{split}
\label{XY4}
\end{align}

\noi
Then, by substituting \eqref{XY4}
(for $\dt \ft w$) in \eqref{XY3}
and 
 applying the 6-linear estimate
from the proof of Proposition 3.4 in \cite{Kwak},  
we bound the contribution  
from this term to \eqref{XY2} 
by 
\begin{align}
&   C\Big( \|u_0\|_{H^\s},  \| u \|_{Y^{\s,\frac{1}{2} }_{u_0, T}}, 
\|u_N \|_{Y^{\s,\frac{1}{2} }_{u_0, T}} \Big) 
 \|u - u_N\|_{Y^{\s,\frac{1}{2} }_{u_0, T}}.
\label{XY5}
\end{align}

As for the second term on the right-hand side of \eqref{X13}, 
we first write
\begin{align} 
\begin{split}
 \P_{\le N}  &  \big( \M(w) - \M(\P_{\le N} w)\big)\\
&  = \P_{\le N} \M(\P_{> N}w, w, w)
+ \P_{\le N} \M(\P_{\le N}w, \P_{> N}w, w)\\
& \quad 
+ \P_{\le N} \M(\P_{\le N}w, \P_{\le N}w, \P_{> N}w).
\end{split}
\label{XY6}
\end{align}

\noi
Namely, one of the factors is given by $\P_{>N} w$.
Then, by substituting \eqref{XY6}
(for $\dt \ft w$) in \eqref{XY3}
and 
 applying the 6-linear estimate
from the proof of Proposition 3.4 in \cite{Kwak} as before,   
we bound the contribution from this term to \eqref{XY2} by 
\begin{align}
&   C\Big( \|u_0\|_{H^\s},  \| u \|_{Y^{\s,\frac{1}{2} }_{u_0, T}} \Big) 
 \|\P_{>N} u \|_{Y^{\s,\frac{1}{2} }_{u_0, T}}.
\label{XY7}
\end{align}

As for the   the third term on the right-hand side of \eqref{X13}:
\[ \P_{> N} \M( w)
=  \P_{> N} \M( w, w, w),\]

\noi
we first note that this term vanishes unless one of the factors has frequencies
greater than $\frac N3$.
 Then, proceeding as above, 
we bound the contribution from this term to \eqref{XY2} by 
\begin{align}
&   C\Big( \|u_0\|_{H^\s},  \| u \|_{Y^{\s,\frac{1}{2} }_{u_0, T}} \Big) 
 \|\P_{>\frac N3} u \|_{Y^{\s,\frac{1}{2} }_{u_0, T}}.
\label{XY8}
\end{align}

\noi
Then, putting \eqref{XY5}, \eqref{XY7}, and \eqref{XY8} together, we obtain
\begin{align}
\begin{split}
\eqref{XY2}
&  \leq 
   C\Big( \|u_0\|_{H^\s},  \| u \|_{Y^{\s,\frac{1}{2} }_{u_0, T}}, 
\|u_N \|_{Y^{\s,\frac{1}{2} }_{u_0, T}} \Big) 
\Big( \|u - u_N\|_{Y^{\s,\frac{1}{2} }_{u_0, T}}+ 
 \|\P_{>\frac N3} u \|_{Y^{\s,\frac{1}{2} }_{u_0, T}}\Big)\\
&  \leq 
   C'\Big(   \| u \|_{Y^{\s,\frac{1}{2}+\eps}_{u_0, T}}, 
\|u_N \|_{Y^{\s,\frac{1}{2}+\eps }_{u_0, T}} \Big) 
\Big( \|u - u_N\|_{Y^{\s,\frac{1}{2}+\eps }_{u_0, T}}+ 
 \|\P_{>\frac N3} u \|_{Y^{\s,\frac{1}{2}+\eps }_{u_0, T}}
\Big)
\end{split}
\label{XY9}
\end{align}

\noi
for any $\eps>0$. Here, 
in the second inequality, we used the embedding \eqref{X7a}
(for the $Y^{\s,\frac{1}{2}+\eps }_{u_0, T}$-space).

 Finally, 
given the  smooth solutions 
  $v$ and $v_N$  
to \eqref{4NLS6} and \eqref{4NLS7}, respectively, 
 with $v|_{t = 0} = v_N|_{t = 0} = u_0 \in C^\infty(\T)$, 
 let $u= \JJ^{-1}(v)$ and $u_N = \JJ^{-1}(v_N)$.
Then, the desired bound \eqref{X11}
follows from \eqref{XY9}
with \eqref{X1}, \eqref{X2}, and \eqref{YY1}.
 This concludes the proof of 
Lemma \ref{LEM:T3}.
\end{proof}

\begin{remark} \label{REM:T4}\rm
As pointed out in \cite{MT}, 
the smoothness assumption in Lemmas \ref{LEM:T2} and \ref{LEM:T3}
is not necessary.
In view of Lemma \ref{LEM:T1}, it suffices to assume that $v, v_N 
\in X^{\s,\frac{1}{2}+\eps }_T$ for $\s > -\frac 12$.
See~\cite{LSZ} for details.
We also point out that,
in Lemmas \ref{LEM:T2} and \ref{LEM:T3},  
 the endpoint $\s = -\frac 13$ is excluded 
so that the estimates in these lemmas  hold for 
{\it rough} solutions 
in $C([0, T]; H^\s(\T))$, $-\frac 13< \s < 0$, 
for {\it any} $T>0$,  using the global-in-time control \eqref{growth1}, 
which is valid only for  $\s > -\frac 13$.

\end{remark}

\subsection{Proof of  Proposition \ref{PROP:app2}} 

We now establish the approximation property of the truncated dynamics \eqref{4NLS4}
(Proposition \ref{PROP:app2}).
In view of the approximation result in $L^2(\T)$ (see \cite{OTz}), 
we restrict our attention to the range $-\frac 13 < \s < 0$.
We first establish the following preliminary lemma.

\begin{lemma}
\label{LEM:app1}
Let $ -\frac{1}{3}<\s<0$ and $u_0\in H^\s(\T)$.
Then, for any $T>0$ and $\dl>0$, there exists $N_0=N_0(T, u_0, \dl)\in \N$ such that
\begin{align*}
\| \Psi(t)(u_0)-\Psi_N(t)(u_0) \|_{H^{\s} }<\dl
\end{align*}

\noi
for any $t\in [0,T]$ and $N\geq N_0$. 
\end{lemma}

\begin{proof}

We first consider the high frequency part of the dynamics.
Recalling that 
 $\P_{>N}\Psi_N(t)(u_{0})=S_{u_0}(t)\P_{>N}u_{0} $, 
 where $S_{u_0}(t)$ is as in \eqref{X6}.
Hence, there exists $N_1=N_1( u_0, \dl) \in \N$ such that 
\begin{align*}
\| \P_{>N}\Psi_N(t)(u_0)  \|_{L^\infty_T H^\s_x}
=\|S_{u_0}(t)\P_{>N}u_{0}\|_{L^\infty_T H^\s_x}
=\| \P_{>N}u_{0} \|_{H^\s}<\frac{\dl}{4}
\end{align*} 

\noi
for any  $N\geq N_1$. 
From \eqref{growth1} with $N = \infty$ and the Lebesgue dominated convergence theorem, 
we have  
\begin{align*}
\| \P_{>N}\Psi(t)(u_0)  \|_{L^\infty_TH^\s_x}
\les\|  \P_{>N}\Psi(t)(u_{0}) \Vert_{X^{\s,\frac{1}{2} +\eps}_T}  <\frac{\dl}{4}
\end{align*}

\noi
for any $N \ge N_2=N_2(T, u_0, \dl) \in \N$.

Hence, 
it suffices to show that there exists $N_3=N_3(T, u_0, \dl) \in \N$ such that
\begin{align}
\|  \P_{\leq N} \Psi (t)(u_0)-\P_{\leq N}\Psi_N(t)(u_0)  \|_{L^\infty_T H^\s_x}<\frac{\dl}{2}
\label{Y4}
\end{align}

\noi
for any $N\geq N_3$.
 By writing \eqref{4NLS6} and \eqref{4NLS7} in the Duhamel formulations
 with $v (t)= \Psi(t) (u_0)$
 and   $v_N (t)= \Psi_N(t) (u_0)$, we have
\begin{align}
\begin{split}
\P_{\leq N}v (t) -\P_{\leq N}v_N (t)
& = -i \sum_{j = 1}^2\int_0^t S(t-t') \big(\P_{\leq N}\NN_j(v) -\P_{\le N} \NN_j^N(v_N)   \big)(t')dt'\\
& = : \1 + \II.
\end{split}
\label{Y5}
\end{align}

We set $w_N = \P_{\le N} v - \P_{\le N} v_N$.
We first estimate $\1$.
From \eqref{X1a} and \eqref{X4}, 
we have 
\begin{equation*}
\begin{split}
\P_{\leq N} & \NN_1(v)-\P_{\leq N}\NN_1^N(v_N)\\
&=\sum_{ | n| \leq N} e^{inx}
\sum_{\G_N(n)}
e^{it\Ta(\bar{n})} \Big(\ft w_N (n_1,t)\cj{\ft v(n_2,t)} \ft v (n_3,t)\\
&\hphantom{X}
 + \ft v_N (n_1,t)\cj{\ft w_N(n_2,t)}\ft v(n_3,t)
+\ft v_N (n_1,t) \cj{\ft v_N(n_2,t)}\ft w_N(n_3,t)  \Big) \\
&\hphantom{X}
+\sum_{ | n| \leq N} e^{inx} 
\sum_{\substack{\G(n)\\ \max\limits_{j = 1, 2, 3}|n_j| > N}}  
e^{it\Ta(\bar n )}\ft v (n_1,t)\cj {\ft v(n_2,t)}\ft v (n_3,t).
\end{split}
\end{equation*}

\noi	
Hence, from Lemmas \ref{LEM:lin} and  \ref{LEM:T1} with \eqref{growth1}, 
we have
\begin{equation}
\begin{split}
\|\1 \|_{X^{\s, \frac{1}{2}+\eps}_\tau}
& \les \tau^\eps \| \P_{\leq N} \NN_1(v)-\P_{\leq N}\NN_1^N(v_N)  \|_{X^{\s,-\frac{1}{2}+ 2\eps}_\tau}\\
&\le \tau^\eps C(T,R)\Big( \| w_N\|_{X^{\s,\frac{1}{2} +\eps}_\tau } 
+ \| \P_{>N}v \|_{X^{\s,\frac{1}{2} +\eps}_\tau}\Big)
\end{split}
\label{Y6}
\end{equation}

\noi
for any $\tau \in [0, T]$, where $R = \|u_0 \|_{H^\s}$.

Next, we consider $\II$ in \eqref{Y5}.
From \eqref{X3} and \eqref{X4}, we have 
\begin{equation}
\begin{split}
 \II & =  i \int_0^t S(t - t') 
\sum_{| n| \leq N}e^{inx}
 \Big(  |\ft v (n,t')|^2-|\ft u_0(n) |^2  \Big)\ft w _N(n,t')
\,  dt'\\
 & \hphantom{X}+  i \int_0^t S(t - t') 
 \sum_{| n| \leq N}e^{inx} \Big( |\ft v (n,t')|^2-|\ft v_N (n,t') |^2      \Big)\ft v _N(n,t') \, dt' \\
&=: \II_1+\II_2.
\end{split}
\label{Y7}
\end{equation}

\noi
By Lemma \ref{LEM:lin}, the fundamental theorem of calculus, \eqref{4NLS6},  Lemma \ref{LEM:T2}, 
and \eqref{growth1}, we have 
\begin{align}
\|  & \II_1 \|_{X^{\s,\frac{1}{2}+ \eps}_\tau} 
 \les \tau^\eps \|  (i \dt - \dx^4) \II_1 \|_{X^{\s,-\frac{1}{2}+ 2\eps}_\tau}
\le \tau^\eps \| (i \dt - \dx^4) \II_1 \|_{L^2_\tau H^\s_x}
\notag \\
& \les \tau^\eps   \sup_{\substack{t \in [0, \tau]\\  | n| \leq N}} 
\bigg| \Re \int_0^t \dt \ft v(n, t')  \cj {\ft v (n, t') } dt' \bigg|
\cdot  
\| w_N \|_{X^{\s,\frac{1}{2}+ \eps}_\tau}
\notag \\
& 
=  \tau^\eps   \sup_{\substack{t \in [0, \tau]\\  | n| \leq N}} 
 \bigg| \Im  \bigg( \int_0^t \sum_{\G(n)}
e^{it \Ta(\bar n)}
\ft v (n_1, t')\cj{\ft v (n_2, t')}\ft v (n_3, t')\cj{\ft v (n, t')}  dt'  \bigg)\bigg|
\cdot  
\| w_N \|_{X^{\s,\frac{1}{2}+ \eps}_\tau}
\notag \\
& \le \tau^\eps C(T,R)\| w_N \|_{X^{\s,\frac{1}{2}+\eps}_\tau} 
\label{Y8}
\end{align} 	

\noi
for any $\tau \in [0, T]$.

Recalling that $v|_{t = 0} = v_N|_{t = 0} = u_0$, 
it follows from 
the fundamental theorem of calculus, \eqref{4NLS6},  \eqref{4NLS7},  
Lemmas \ref{LEM:T3} and \ref{LEM:T2}, and \eqref{growth1}
that 
\begin{align}
\Big||\ft v  & (n,t)|^2-|\ft v_N (n,t) |^2     \Big|
 \leq \Big||\ft v (n,t)|^2-|\ft u_0(n)|^2 \Big| +  \Big||\ft v_N (n,t) |^2     -|\ft u_0(n)|^2\Big|
\notag \\
& \le 2
\bigg| \Im  \bigg( \int_0^t \sum_{\G_N(n)}
e^{it \Ta(\bar n)}
\Big( 
\ft v (n_1, t')\cj{\ft v (n_2, t')}\ft v (n_3, t')\cj{\ft v (n, t')}
\notag \\
& \hphantom{XXXXXXXXXX}
- 
\ft v_N (n_1, t')\cj{\ft v_N (n_2, t')}\ft v_N (n_3, t')\cj{\ft v_N (n, t')} \Big)  dt'  \bigg)\bigg|
\notag \\
& + 2
\bigg| \Im  \bigg( \int_0^t 
\sum_{\substack{\G(n)\\ \max\limits_{j = 1, 2, 3}|n_j| > N}}  
e^{it \Ta(\bar n)}
\ft v (n_1, t')\cj{\ft v (n_2, t')}\ft v (n_3, t')\cj{\ft v (n, t')}dt'   \bigg)\bigg| \notag \\
& \leq  C(T, R) 
\Big( \|w_N\|_{X^{\s,\frac{1}{2}+\eps }_\tau} + \|\P_{> \frac N3} v \|_{X^{\s,\frac{1}{2}+\eps }_\tau} \Big), 
\label{Y9}
\end{align}

\noi
uniformly in  $| n| \leq N$ and $0 \le t \le \tau \le T$.
Then, from   \eqref{Y7},   Lemma \ref{LEM:lin},  \eqref{Y9}, 
and \eqref{growth1}, we obtain 
\begin{align}
\begin{split}
\|    \II_2 \|_{X^{\s,\frac{1}{2}+ \eps}_\tau } 
&  \les \tau^\eps \|  (i \dt - \dx^4) \II_2 \|_{X^{\s,-\frac{1}{2}+ 2\eps}_\tau }
\le \tau^\eps \| (i \dt - \dx^4) \II_2 \|_{L^2_\tau H^\s_x}\\
& \les \tau^\eps   \sup_{\substack{t \in [0, \tau ]\\  | n| \leq N}} 
\Big| |\ft v (n,t)|^2-|\ft v_N (n,t) |^2      \Big|
\cdot  
\| v_N \|_{X^{\s,\frac{1}{2}+ \eps}_\tau}\\
& \le \tau ^\eps C(T,R)
\Big( \|w_N\|_{X^{\s,\frac{1}{2}+\eps }_\tau } + \|\P_{> \frac N3} v \|_{X^{\s,\frac{1}{2}+\eps }_\tau } \Big).
\end{split}
\label{Y10}
\end{align}

Therefore, from
\eqref{Y5}, 
\eqref{Y6}, \eqref{Y7},
\eqref{Y8},  and \eqref{Y10}, 
we have
\begin{align}
\| w_N  \|_{X^{\s,\frac{1}{2}+\eps }_\tau}
& \les \tau ^\eps C_*(T,R)
\Big( \|w_N\|_{X^{\s,\frac{1}{2}+\eps }_\tau } + \|\P_{> \frac N3} v \|_{X^{\s,\frac{1}{2}+\eps }_\tau } \Big)
\label{Y11}
\end{align}

\noi
for any $\tau \in [0, T]$.
By choosing $\tau=\tau(T, R) >0$ sufficiently small
such that 
\begin{align}
\tau ^\eps C_*(T,R) \leq \frac 12 , 
\label{Y11a}
\end{align}

\noi
we obtain,  from \eqref{Y11} with \eqref{X7a}, 
\begin{align}
\|  w_N \|_{L^\infty_\tau H^\s_x} 
\les \|  w_N  \|_{X^{\s,\frac{1}{2}+\eps }_\tau}
\leq C_1(T,R) 
\| \P_{> \frac N 3}v \|_{X^{\s,\frac{1}{2}+\eps}_T}.
\label{Y12}
\end{align}

We now consider the second time interval $I_2 = [\tau, 2\tau]$.
The estimates 
\eqref{Y6} on $\1$
and \eqref{Y8} on $\II_1$ also hold  on $[\tau, 2\tau]$.
As for the analysis on $\II_2$,  we need to make the following modification
in~\eqref{Y9}.
By writing
\begin{align}
\begin{split}
|\ft v   (n,t)|^2-|\ft v_N (n,t) |^2     
&  = \Big(|\ft v (n,t)|^2-|\ft v(n, \tau)|^2 \Big) - \Big(|\ft v_N (n,t) |^2 
    -|\ft v_N(n, \tau)|^2\Big) \\
& \hphantom{X} + 
\Big(|\ft v   (n,\tau)|^2-|\ft v_N (n,\tau) |^2 \Big),     
\end{split}
\label{Y13}
\end{align}

\noi
we estimate the first two terms on the right-hand side of \eqref{Y13}
by using the fundamental theorem of calculus as in \eqref{Y9}, while  the last term 
on the right-hand side of \eqref{Y13} is already controlled by \eqref{Y9}
with $t = \tau$.
Together with \eqref{Y12}, this gives 
\begin{align}
\begin{split}
|\ft v   (n,t)|^2-|\ft v_N (n,t) |^2     
& \leq  C(T, R) 
\Big( \|w_N\|_{X^{\s,\frac{1}{2}+\eps }([\tau, 2\tau])} 
+ \|\P_{> \frac N3} v \|_{X^{\s,\frac{1}{2}+\eps }([\tau, 2\tau])} \Big)\\
&  \hphantom{X}
+  C_1'(T, R)\|\P_{> \frac N3} v \|_{X^{\s,\frac{1}{2}+\eps }_\tau} 
\end{split}
\label{Y14}
\end{align}

\noi
uniformly in  $| n| \leq N$ and $0 \le \tau \le t \le 2\tau \le T$.
Therefore, proceeding as before with \eqref{Y14}, we have
\begin{align}
\| w_N  \|_{X^{\s,\frac{1}{2}+\eps }([\tau, 2\tau])}
& \les \tau ^\eps C_*(T,R)
 \|w_N\|_{X^{\s,\frac{1}{2}+\eps }([\tau, 2\tau])} 
+ 
\tau^\eps C_1''(T,R)
\|\P_{> \frac N3} v \|_{X^{\s,\frac{1}{2}+\eps }_T}.
\label{Y15}
\end{align}

\noi
Hence, from \eqref{Y15} and \eqref{Y11a}, we obtain 	
\begin{align*}
\|  w_N \|_{L^\infty([\tau, 2\tau]; H^\s)} 
\les \|  w_N  \|_{X^{\s,\frac{1}{2}+\eps }([\tau, 2\tau])}
\leq C_2(T,R) 
\| \P_{> \frac N 3}v \|_{X^{\s,\frac{1}{2}+\eps}_T}.
\end{align*}

\noi
By repeating this argument, 
we have
\begin{align*}
\|  w_N \|_{L^\infty(I_j; H^\s)} 
\les \|  w_N  \|_{X^{\s,\frac{1}{2}+\eps }(I_j)}
\leq C_j(T,R) 
\| \P_{> \frac N 3}v \|_{X^{\s,\frac{1}{2}+\eps}_T}
\end{align*}

\noi
on the $j$th time interval
$I_j = [(j-1)\tau, j\tau]\cap [0, T]$.
Note that
while $C_j(T, R)$ is increasing in $j$, 
it follows from 
our choice of $\tau$ in \eqref{Y11a} 
that 
$\max_{j = 1, \dots, [\frac{T}{\tau}]+1} C_j (T, R) \leq C^*(T, R)$ for some $C^*(T, R)>0$.
Therefore, we conclude that 
\begin{align}
\|  w_N \|_{L^\infty_T H^\s_x} 
\leq C^*(T,R) 
\| \P_{> \frac N 3}v \|_{X^{\s,\frac{1}{2}+\eps}_T}.
\label{Y18}
\end{align}

\noi
Then, the desired bound 
\eqref{Y4}
follows from \eqref{Y18} and 
the Lebesgue dominated convergence theorem
with \eqref{growth1}.
This completes the proof of Lemma \ref{LEM:app1}.
\end{proof}

\begin{remark}\label{rem:app2}
\rm

Due to the lack of local uniform continuity of the solution map in negative Sobolev spaces, 
it is crucial that $\Psi(t)(u_0)$ and $\Psi_N(t)(u_0)$ have the same initial condition 
$u_0$ in the proof of Lemma \ref{LEM:app1};
 see \eqref{Y9}.
\end{remark}

We conclude this section by presenting the proof of
 Proposition \ref{PROP:app2}.
 We follow  \cite[Proposition~2.10]{Tz} and \cite[Proposition B.3/6.21]{OTz}.

\begin{proof}[Proof of Proposition \ref{PROP:app2}]
Let $u_0\in A $, $t \in \R$, and small $\dl > 0$.
Write 
\[\Phi(t) (u_0) =\Phi_N(t)(\Phi_N(-t)\Phi(t)(u_0)  ).\]

\noi
By setting 
$w_N=\Phi_N(-t)\Phi(t)(u_0)$, 
 it suffices to show that there exists 
$ N_0(t,R, u_0, \dl)\in \N $
such that 
\begin{align*}
w_N\in A+B_{\dl}
\end{align*}
for every $N\geq N_0$.
Define $z_N$ by 
\begin{align*}
 z_N=\Phi_N(-t)\Phi(t)(u_0)-u_0
\end{align*}

\noi
such that $w_N=u_0+z_N$.
Since $u_0 \in A$, 
we only need to 
 check that   $z_N \in B_{\dl }$ for all $N \gg 1$.
 By writing
\begin{align*}
z_N=\Phi_N(-t)\big(\Phi(t)(u_0)-\Phi_N(t)(u_0)\big), 
\end{align*}

\noi
it follows
 from the uniform (in $N$) growth bound on the $H^\s$-norm
of solutions to \eqref{4NLS4}
(see \cite[Proposition 6.6]{OW1} for the case $N = \infty$) 
that
\begin{align*}
\|  z_N \|_{H^\s}
& = \big\|\Phi_N(-t)\big(\Phi(t)(u_0)-\Phi_N(t)(u_0)\big) \big\|_{H^\s}\\
&  
\leq C(t)
\|\Phi(t)(u_0)-\Phi_N(t)(u_0)\|_{H^\s}^{c(\s)}
\end{align*}

\noi
for some $c(\s) > 0$.
By the unitarity of the gauge transform $\JJ$ in \eqref{X1} (for fixed $t \in \R$)
and 
Lemma \ref{LEM:app1}, we have 
\begin{align*}
\|\Phi(t)(u_0)-\Phi_N(t)(u_0)\|_{H^\s} \too 0
\end{align*}

 \noi
 as $N \to \infty$.
 This implies that $z_N \in B_\dl$
 for $N \ge  N_0(t,R, u_0, \dl)\in \N $.
 This proves Proposition~\ref{PROP:app2}.
\end{proof}

\section{Normal form reductions}\label{SEC:NF}

In this section, we present the proof of Proposition \ref{PROP:energy} and Lemma \ref{LEM:E2}
by implementing an infinite iteration 
of normal form reductions as in \cite{OST, OW1}.
This procedure allows us to construct an infinite sequences of correction terms
and thus build  the desired modified energies $\EE_{N}(u)(t)$ and $\EE(u)(t)$ in \eqref{E2}.

\subsection{Main proposition}

In this subsection, 
by expressing the multilinear terms in the series expansion \eqref{E1}
in terms of the interaction representation, 
we state 
the bounds on these multilinear terms
 (Proposition \ref{PROP:energy2}).
By  assuming these bounds, 
we then present the proofs of  
 Proposition~\ref{PROP:energy} and Lemma~\ref{LEM:E2}.


In order to encode multilinear dispersion in an effective manner, 
it is convenient to work with the following interaction representation of $u$
defined by 
\begin{align*}
v(t):=S(-t)u(t).
\end{align*}

\noi 
On the Fourier side, we have
\begin{align*}
v_n(t)=e^{itn^4}u_n(t), 
\end{align*}

\noi 
where, for simplicity of notation, 
we set $v_n(t) = \ft v(n, t)$, etc.
We use this short-hand notation in the remaining part of this section.
If $u$ is a solution to \eqref{4NLS2}, 
then
 $\{v_n\}_{n \in \Z}$ satisfies the following equation:
\begin{align}
\begin{split}
\dt v_n 
& = -i \sum_{\G(n)} e^{-i \phi(\bar n) t} v_{n_1}\cj{v_{n_2}}v_{n_3}
+ i |v_n|^2 v_n \\
& =: \NN(v)_n + \RR(v)_n, 
\end{split}
\label{4NLS8}
\end{align}

\noi
where  $\phi(\bar n)$ and  $\G(n)$ are as in \eqref{phi1} and \eqref{G1}.  
By writing 
\eqref{4NLS5} in terms of the interaction representation, we have 
 the following finite dimensional system of ODEs:
\begin{align}
	\dt v_n = 
	-i 
	\sum_{\G_N(n)} e^{-i \phi(\bar n) t} v_{n_1}\cj{v_{n_2}}v_{n_3}
	+ i |v_n|^2 v_n,  \qquad |n| \leq N
	\label{4NLS10}
\end{align}

\noi
with $v|_{t = 0} = \P_{\leq N}v|_{t = 0}$,
namely, $v_n|_{t=0} = 0$ for $|n| > N$.

In the following, 
we simply say that $v$ is a solution 
to \eqref{4NLS10} 
if  $v$ is  a solution to \eqref{4NLS10} when $N \in \N$
and  to \eqref{4NLS8} when $N = \infty$. 
We state out main result in this section.

\begin{proposition}\label{PROP:energy2}
Let $\frac 3{10}< s \leq \frac 12$
and $\s = s - \frac 12 - \eps$ for some small $\eps > 0$. 
Then, given  $N \in \N \cup\{\infty\}$, 
there exist multilinear forms 
$\big\{\textup{\Ns}_{0, N}^{(j)}(t)\big\}_{j = 2}^\infty$, 
$\big\{\textup{\Ns}_{1, N}^{(j)}(t)\big\}_{j = 2}^\infty$, 
and 
$\big\{\textup{\Rs}_N^{(j)}(t)\big\}_{j = 2}^\infty$, 
depending on $t\in \R$,  
such that
\begin{align} 
\frac {d}{dt} \bigg(\frac 12 \| v (t) \|_{H^s}^2\bigg)
=  \frac {d}{dt}\bigg( \sum_{j = 2}^\infty \textup{\Ns}^{(j)}_{0, N}(t)(v(t))\bigg)
+ \sum_{j = 2}^\infty \textup{\Ns}^{(j)}_{1, N}(t)(v(t)) + \sum_{j = 2}^\infty \textup{\Rs}_N^{(j)}(t)(v(t))
\label{E8}
\end{align}
	
\noi
for any solution  $v \in C(\R; H^\s(\T))$ to \eqref{4NLS10}.\footnote{Note that the left-hand side of \eqref{E8}
does not a priori make sense for $v \in C(\R; H^\s(\T))$.
The identity \eqref{E8} is to be understood in the limiting sense for rough solutions.}
Here, 
$\textup{\Ns}_{0, N}^{(j)}(t)$ are $2j$-linear forms, 
while
$\textup{\Ns}_{1, N}^{(j)}$ and   $\textup{\Rs}_N^{(j)}$  are $(2j+2)$-linear forms, satisfying the following bounds in $H^\s(\T)$\textup{;}
there exist positive constants $C_0(j)$, $C_1(j)$, and $C_2(j)$, 
decaying faster than any exponential  rate\footnote{In fact, 
by slightly modifying the proof,   we can make  $C_0(j)$, $C_1(j)$, and $C_2(j)$
decay as fast as we want as $j \to \infty$.}
as $j \to \infty$ 
such that 
\begin{align} 
\sup_{t \in \R}
\big| \textup{\Ns}^{(j)}_{0, N}(t)(f_1, \dots, f_{2j})\big|
& \le
C_0(j)  \prod_{k= 1}^{2j}\|f_k\|_{H^\s},
\label{E9} 
\\
\sup_{t \in \R}
\big| \textup {\Ns}^{(j)}_{1, N}(t)(f_1, \dots, f_{2j+2})\big| 
& \le
C_1(j) \prod_{k= 1}^{2j+2}\|f_k\|_{H^\s}, 
\label{E10}
\\
\sup_{t \in \R}
\big| \textup {\Rs}^{(j)}_N(t)(f_1, \dots, f_{2j+2})\big|
&  \le
C_2(j)  \prod_{k= 1}^{2j+2}\|f_k\|_{H^\s}
\label{E11}
\end{align}

\noi
for $j = 2, 3, \dots$.
Note that these constants
$C_0(j)$, $C_1(j)$, and $C_2(j)$
are  independent of the cutoff size $N \in \N \cup \{\infty\}$.

\end{proposition}

We now present the proofs of Proposition \ref{PROP:energy} and Lemma \ref{LEM:E2}
by assuming Proposition \ref{PROP:energy2}.
First, we prove Proposition \ref{PROP:energy}. 
Given $N \in \N \cup\{\infty\}$, let $u \in C(\R; H^\s(\T))$
be a  solution  to \eqref{4NLS5},
 satisfying  the growth bound \eqref{E3b}.
Then, we define the multilinear form 
$\NN^{(j)}_{0, N}$, 
$ \NN^{(j)}_{1, N}$, and 
$\RR^{(j)}_N$
by setting
\begin{align}
\begin{split}
\NN^{(j)}_{0, N}(u(t)) & := 	\Ns^{(j)}_{0, N}(t)(S(-t)u(t)), \\
 \NN^{(j)}_{1, N}(u(t))&:= 	\Ns^{(j)}_{1, N}(t)(S(-t)u(t)),\\ 
 \RR^{(j)}_N(u(t)) & : = \Rs^{(j)}_N(t)(S(-t)u(t)).
\end{split}
\label{ZZ1}
\end{align}

\noi
While the multilinear forms 
$\Ns^{(j)}_{0, N}$, $\Ns^{(j)}_{1, N}$, and $\Rs^{(j)}_N$ 
appearing in  Proposition~\ref{PROP:energy2}
are  
non-autonomous (i.e.~they depend on $t \in \R$), 
it is  easy to see from the construction of 
these multilinear forms carried out in the remaining part of this section
that 
the multilinear forms
$\NN^{(j)}_{0, N}$, 
$ \NN^{(j)}_{1, N}$, and 
$\RR^{(j)}_N$ defined in \eqref{ZZ1}
are indeed  autonomous.

From \eqref{E8} and \eqref{ZZ1} with the unitarity 
of $S(t)$, we obtain \eqref{E1}.
By defining the modified energy $\EE_N(u)$ 
as in \eqref{E2}, 
it follows from \eqref{E1} and \eqref{ZZ1}
\begin{align} 
\frac{d}{dt} \EE_N(u)(t) 
=  \sum_{j = 2}^\infty \Ns^{(j)}_{1, N}(t)(S(-t) u(t)) + \sum_{j = 2}^\infty \Rs_N^{(j)}(t)(S(-t)u(t)).
\label{ZZ2}
\end{align}

\noi
Then, from \eqref{ZZ2}
and 
Proposition \ref{PROP:energy2}
 together with the growth bound \eqref{E3b}
and the fast decay (in $j$)
of the constants
$C_0(j)$, $C_1(j)$, and $C_2(j)$, we obtain 
\begin{align*} 
\sup_{t \in [0, T]}\bigg| \frac{d}{dt} \EE_{N}(u)(t)  \bigg|
& \le
\sum_{j=2}^\infty \big(C_1(j) + C_2(j)\big)  R^{2j+2} \\
& \le C_s(R).
\end{align*}

\noi
This proves Proposition \ref{PROP:energy}.

We now turn to the proof of Lemma \ref{LEM:E2}.	
Let $u \in B_R \subset H^\s(\R)$. 
Then, from \eqref{E3a}, \eqref{ZZ1}, and~\eqref{E9} in Proposition \ref{PROP:energy2}, 
we have
\begin{align*}
| \Sf_N(u)|
& = \bigg|\sum_{j=2}^{\infty} \NN^{(j)}_{0,N}(\P_{\leq N}u )\bigg|
= \bigg| \sum_{j=2}^{\infty} \Ns^{(j)}_{0,N}(t) (\P_{\le N} S(-t) u) \bigg|\\
& \leq   \sum_{j=2}^{\infty} C_0(j) R^{2j}  \leq C_s(R) 
\end{align*}

\noi
for any $N \in \N\cup\{\infty\}$
(and any $t \in \R$). 
As for the convergence part, 
we refer the readers to Subsection 4.7 in \cite{OST} for details.
 This completes
 the proofs of Proposition \ref{PROP:energy} and Lemma \ref{LEM:E2}.

 \begin{remark}\label{REM:NF}\rm
	
In \cite{OST}, Proposition \ref{PROP:energy2}
was shown for $\s = 0$ (and $\frac 12 < s < 1$), where
the divisor counting argument played an important role.
In the current setting with $\s < 0$, 
we need to make use of the fourth order dispersion
to gain derivatives 
and,  for this purpose, we follow the argument in \cite{OW1}.
In particular,  we do not rely on the divisor counting argument.
The essential difference between our argument and that in~\cite{OW1}
is the presence of the weight $\jb{n}^{2s}$, 
coming from the $H^s$-norm squared on the left-hand side of~\eqref{E8}. 
Namely, for our problem, we need to exhibit a stronger smoothing property 
than that in \cite{OW1}, resulting in a worse regularity restriction $\s > - \frac 15$
in Proposition \ref{PROP:energy2}.

\end{remark}

\subsection{Notations: index by ordered bi-trees}
\label{SUBSEC:tree}

In this subsection, 
we go over notations from \cite{GKO, OST, OW1}
 for implementing an infinite iteration of normal form reductions.
 Our main goal is to apply
 normal form reductions  to the $H^s$-energy functional\footnote{More precisely, 
to the evolution equation satisfied by the $H^s$-energy functional.}
and thus
 we need tree-like structures that grow in two directions.
 For our analysis, 
 ordered bi-trees in Definition \ref{DEF:tree3}
 play an essential role.

\begin{definition} \label{DEF:tree1} \rm
(i) 
Given a partially ordered set $\TT$ with partial order $\leq$, 
we say that $b \in \TT$ 
with $b \leq a$ and $b \ne a$
is a child of $a \in \TT$,
if  $b\leq c \leq a$ implies
either $c = a$ or $c = b$.
If the latter condition holds, we also say that $a$ is the parent of $b$.

\smallskip

\noi
(ii) A tree $\TT $ is a finite partially ordered set satisfying
the following properties:
\begin{enumerate}
		
\item[(a)] Let $a_1, a_2, a_3, a_4 \in \TT$.
If $a_4 \leq a_2 \leq a_1$ and  
$a_4 \leq a_3 \leq a_1$, then we have $a_2\leq a_3$ or $a_3 \leq a_2$,
		
\item[(b)]
A node $a\in \TT$ is called terminal, if it has no child.
A non-terminal node $a\in \TT$ is a node 
with  exactly three ordered\footnote{For example, 
we simply label the three children as $a_1, a_2$, and $a_3$
by moving from left to right in the planar graphical representation of the tree $\TT$.
As we see below, we assign the Fourier coefficients of the interaction representation $v$ at $a_{1}$ and $a_{3}$, 
while we assign the complex conjugate of
the Fourier coefficients of $v$ at the second child $a_{2}$.} children denoted by $a_1, a_2$, and $a_3$,

\item[(c)] There exists a maximal element $r \in \TT$ (called the root node) such that $a \leq r$ for all $a \in \TT$,
		
\item[(d)] $\TT$ consists of the disjoint union of $\TT^0$ and $\TT^\infty$,
where $\TT^0$ and $\TT^\infty$
denote  the collections of non-terminal nodes and terminal nodes, respectively.
\end{enumerate}

\smallskip
	
\noi
(iii) A {\it bi-tree} $\TT = \TT_1 \cup \TT_2$ is 
a disjoint union of two trees $\TT_1$ and $\TT_2$,
where the root nodes $r_j$ of $\TT_j$, $j = 1, 2$,  are joined by an edge.
A bi-tree
$\TT$ consists of the disjoint union of $\TT^0$ and $\TT^\infty$,
where $\TT^0$ and $\TT^\infty$
denote  the collections of non-terminal nodes and terminal nodes, respectively.
By convention, we assume that the root node $r_1$ of the tree $\TT_1$ is non-terminal,
while the root node $r_2$ of the tree $\TT_2$ may be terminal.

\smallskip
	
\noi
(iv) Given a bi-tree $\TT = \TT_1 \cup \TT_2$, 
we define a projection $\Pi_j$, $j = 1, 2$, onto  a tree
by setting 
\begin{align*}
\Pi_j(\TT) = \TT_j.
\end{align*}

\end{definition}

Note that the number $|\TT|$ of nodes in a bi-tree $\TT$ is $3j+2$ for some $j \in \mathbb{N}$,
where $|\TT^0| = j$ and $|\TT^\infty| = 2j + 2$.
Let us denote  the collection of trees in the $j$th generation 
(namely,  with $j$ parental nodes) by $BT(j)$, i.e.
\begin{equation*}
BT(j) := \{ \TT : \TT \text{ is a bi-tree with } |\TT| = 3j+2 \}.
\end{equation*}

Next, we recall  the  notion of ordered bi-trees,
for which we  keep track of how a bi-tree ``grew''
into a given shape.

\begin{definition} \label{DEF:tree3} \rm
(i) We say that a sequence $\{ \TT_j\}_{j = 1}^J$  is a chronicle of $J$ generations, 
if 
\begin{enumerate}
\item[(a)] $\TT_j \in BT(j)$ for each $j = 1, \dots, J$,
\item[(b)]  $\TT_{j+1}$ is obtained by changing one of the terminal
nodes in $\TT_j$ into a non-terminal node (with three children), $j = 1, \dots, J - 1$.
\end{enumerate}
	
\noi
Given a chronicle $\{ \TT_j\}_{j = 1}^J$ of $J$ generations,  
we refer to $\TT_J$ as an {\it ordered bi-tree} of the $J$th generation.
We denote the collection of the ordered trees of the $J$th generation
by $\mathfrak{BT}(J)$.
Note that the cardinality of $\mathfrak{BT}(J)$ is given by 
$ |\mathfrak{BT}(1)| = 1$ and 
\begin{equation} 
\label{cj1}
|\mathfrak{BT}(J)| = 4\cdot 6 \cdot 8 \cdot \cdots \cdot 2J 
= 2^{J-1}   \cdot J!=: c_J,
\quad J \geq 2.
\end{equation}

\smallskip
	
\noi
(ii) Given an ordered bi-tree $\TT_J \in \mathfrak{BT}(J)$ as above, 
we define  projections $\pi_j$, $j = 1, \dots, J-1$, 
onto the previous generations
by setting
\begin{align*}
\pi_j(\TT_J) = \TT_j \in \mathfrak{BT}(j).
\end{align*}

\end{definition}

We stress that the notion of ordered bi-trees comes with associated chronicles.
For example, 
given two ordered bi-trees $\TT_J$ and $\wt{\TT}_J$
of the $J$th generation, 
it may happen that $\TT_J = \wt{\TT}_J$ as bi-trees (namely as planar graphs) 
according to Definition \ref{DEF:tree1},
while $\TT_J \ne \wt{\TT}_J$ as ordered bi-trees according to Definition \ref{DEF:tree3}.
In the following, when we refer to an ordered bi-tree $\TT_J$ of the $J$th generation, 
it is understood that there is an underlying chronicle $\{ \TT_j\}_{j = 1}^J$.

\smallskip

Given a bi-tree $\TT$, 
we associate each terminal node $a \in \TT^\infty$ with the Fourier coefficient (or its complex conjugate) of the interaction representation 
$v$ and sum over all possible frequency assignments.
For this purpose, we recall the notion of
index functions, 
assigning integers to {\it all} the nodes in $\TT$ in a consistent manner.

\begin{definition} \label{DEF:tree4} \rm
(i) Given  a bi-tree $\TT = \TT_1\cup \TT_2$, 
we define an index function ${\bf n}: \TT \to \mathbb{Z}$ such that
\begin{itemize}
		
\item[(a)] $n_{r_1} = n_{r_2}$, where $r_j$ is the root node of the tree $\TT_j$, $j = 1, 2$,
		
\item[(b)] $n_a = n_{a_1} - n_{a_2} + n_{a_3}$ for $a \in \TT^0$,
where $a_1, a_2$, and $a_3$ denote the children of $a$,
		
\item[(c)] $\{n_a, n_{a_2}\} \cap \{n_{a_1}, n_{a_3}\} = \emptyset$ for $a \in \TT^0$, 
		
\end{itemize}

\noi
where  we identified ${\bf n}: \TT \to \mathbb{Z}$ 
with $\{n_a \}_{a\in \TT} \in \mathbb{Z}^\TT$. 
We use 
$\mathfrak{N}(\TT) \subset \mathbb{Z}^\TT$ to denote the collection of such index functions ${\bf n}$
on $\TT$.
	
\smallskip
	
\noi
(ii) Given a tree $\TT$, we also define 
an index function ${\bf n}: \TT \to \mathbb{Z}$ 
by omitting the condition (a)
and denote by 
$\mathfrak{N}(\TT) \subset \mathbb{Z}^\TT$  the collection of index functions ${\bf n}$
on $\TT$.

\end{definition}

\begin{remark} \label{REM:terminal}
\rm 
(i) In view of the consistency condition (a), 
we can refer to $n_{r_1} = n_{r_2}$
as the frequency  at the root node without ambiguity.
We shall  simply denote it by $n_r$ in the following.

\smallskip
	
\noi
(ii) Given a bi-tree $\TT \in BT(J)$ and $n \in \Z$, consider
the summation over all possible frequency assignments
$\{ \bn \in \mathfrak{N}(\TT): n_r = n\}$.
While $|\TT^\infty| = 2J + 2$, 
there are $2J$ free variables in this summation.
Namely, the condition $n_r = n$ reduces two summation variables.
It is easy to see this by separately considering
the cases $\Pi_2(\TT) = \{r_2\}$
and $\Pi_2(\TT) \ne \{r_2\}$.
\end{remark}

\medskip

Given an ordered bi-tree 
$\TT_J$ of the $J$th generation with a chronicle $\{ \TT_j\}_{j = 1}^J$ 
and associated index functions ${\bf n} \in \mathfrak{N}(\TT_J)$,
 we use superscripts to denote such generations of frequencies.

Fix ${\bf n} \in \mathfrak{N}(\TT_J)$.
Consider $\TT_1 = \pi_1 (\TT_J)$ of the first generation.
Its nodes consist of the two root nodes $r_1$, $r_2$, 
and the children $r_{11}, r_{12}, $ and $r_{13}$ of the first root node $r_1$. 
We define the first generation of frequencies by
\[\big(n^{(1)}, n^{(1)}_1, n^{(1)}_2, n^{(1)}_3\big) :=(n_{r_1}, n_{r_{11}}, n_{r_{12}}, n_{r_{13}}).\]

\noi
The ordered bi-tree $\TT_2 = \pi_2(\TT_J)$ of the second generation 
is constructed from $\TT_1$ by
changing one of its terminal nodes $a \in \TT^\infty_1 = \{ r_2, r_{11}, r_{12}, r_{13}\}$ 
into a non-terminal node.
Then, we define
the second generation of frequencies by setting
\[\big(n^{(2)}, n^{(2)}_1, n^{(2)}_2, n^{(2)}_3\big) :=(n_a, n_{a_1}, n_{a_2}, n_{a_3}).\]

\noi
As we see below, this 
corresponds to introducing a new set of frequencies
after the first differentiation by parts.

In general, we construct  an ordered bi-tree $\TT_j = \pi_j (\TT_J)$ 
of the $j$th generation from $\TT_{j-1}$ by
changing one of its terminal nodes $a  \in \TT^\infty_{j-1}$
into a non-terminal node.
Then, we define
the $j$th generation of frequencies by
\[\big(n^{(j)}, n^{(j)}_1, n^{(j)}_2, n^{(j)}_3\big) :=(n_a, n_{a_1}, n_{a_2}, n_{a_3}).\]

We denote by  $\phi_j$  the 
phase function 
for the frequencies introduced at the $j$th generation:
\begin{align*}
\phi_j & =  \phi_j \big( n^{(j)},  n^{(j)}_1, n^{(j)}_2, n^{(j)}_3\big)
:=   \big(n_1^{(j)}\big)^4 - \big(n_2^{(j)}\big)^4 + \big(n_3^{(j)}\big)^4
- \big(n^{(j)}\big)^4.
\end{align*}

\noi
Note that we have $|\phi_j| \geq 1$ in view of Definition \ref{DEF:tree4} and \eqref{phi2}.
We also denote by $\mu_j$
the  phase function corresponding to the usual cubic NLS (at the $j$th generation):
\begin{align*}
\mu_j & =  \mu_j \big( n^{(j)},  n^{(j)}_1, n^{(j)}_2, n^{(j)}_3\big)
:=   \big(n_1^{(j)}\big)^2 - \big(n_2^{(j)}\big)^2 + \big(n_3^{(j)}\big)^2
- \big(n^{(j)}\big)^2\notag \\
& =-2 \big(n^{(j)} - n_1^{(j)}\big) \big(n^{(j)} - n_3^{(j)}\big).
\end{align*}

\noi
Then, from \eqref{phi2}, we have 
\begin{align}
|\phi_j| \sim (n^{(j)}_\text{max})^2
\cdot| \big(n^{(j)} - n_1^{(j)}\big) \big(n^{(j)} - n_3^{(j)}\big)|
\sim (n^{(j)}_\text{max})^2 \cdot |\mu_j|,
\label{MU2}
\end{align}

\noi
where $n^{(j)}_\text{max}$ is defined by 
\[n^{(j)}_\text{max}: = \max\big(|n^{(j)}|, 
|n_1^{(j)}|, |n_2^{(j)}|, |n_3^{(j)}| \big).\]

Lastly, given  an ordered bi-tree $\TT \in \mathfrak{BT}(J)$ for some $J \in \N$, 
define $A_j \subset \mathfrak{N}(\TT)$ by 
\begin{equation} \label{Cj} 
A_j = \big\{ |\wt{\phi}_{j+1}| \les (2J+4)^{3} |\wt{\phi}_j|\big\} 
\cup \big\{ |\wt{\phi}_{j+1}| \les (2J+4)^{3} |\phi_1|\big\} , 
\end{equation}

\noi
where $\wt{\phi}_j$  is defined by 
\begin{align}
\wt{\phi}_j = \sum_{k = 1}^j \phi_k.
\label{Cj2}
\end{align}

In Subsections \ref{SUBSEC:NF1} and \ref{SUBSEC:NF2}, 
we perform normal form reductions in an iterative manner.
At each step, 
we divide multilinear forms into
nearly resonant part (corresponding to the frequencies belonging to $A_j$) and highly non-resonant part
(corresponding to the frequencies belonging to $A_j^c$)
and apply a normal form reduction only to the highly non-resonant part.
Then, we prove the multilinear estimates 
\eqref{E9}, \eqref{E10}, and \eqref{E11}
for 
a solution $v$ to~\eqref{4NLS10}, 
uniformly in $N \in \N\cup\{\infty\}$.
For simplicity of presentation, however, 
we only consider the $N = \infty$ case 
and work on the equation~\eqref{4NLS8}
without the frequency cutoff $\ind_{|n|\leq N}$
in the following.
We point out that the same normal form reductions
and estimates hold
for the truncated equation \eqref{4NLS10}, uniformly in $N \in \N$, 
with  straightforward modifications:
(i) set 
$\ft v_n = 0$ for all $|n| >N$
and 
(ii) the multilinear forms for \eqref{4NLS10}
are obtained by inserting 
the frequency cutoff $\ind_{|n|\leq N}$ in appropriate places.\footnote{Using the bi-tree notation, 
it follows from \eqref{4NLS10} that 
we simply need to insert the frequency cutoff $\ind_{|n^{(j)}|\leq N}$
on the parental frequency $n^{(j)}$ assigned to each non-terminal node $a \in \TT^0$.}
In the following, we introduce multilinear forms
such as 
$\Ns_0^{(j)}$, $\Ns_1^{(j)}$, $\Ns_2^{(j)}$, and $\Rs^{(j)}$
for the untruncated equation \eqref{4NLS8}.
With a small modification, 
these multilinear forms give rise to
$\Ns_{0, N}^{(j)}$, $\Ns_{1, N}^{(j)}$, $\Ns_{2,N}^{(j)}$, and  $\Rs_N^{(j)}$,  
$N \in \N$,  
for the truncated equation \eqref{4NLS10},
appearing in Proposition \ref{PROP:energy2}.

We point out that 
given finite $N \in \N$, 
a solution to the truncated equation \eqref{4NLS10} is smooth
and therefore 
the formal computations presented in 
Subsections \ref{SUBSEC:NF1} and \ref{SUBSEC:NF2}
can be easily justified for solutions to \eqref{4NLS10}.
When $N = \infty$, 
 we need to impose  the regularity condition $v \in C(\R; H^\s(\T))$, $\s \ge \frac 16$, 
to justify the normal form procedure.
See \cite{GKO, OW1} for details.
Hence, given a solution 
$v \in C(\R; H^\s(\T))$ to~\eqref{4NLS8}
with $- \frac 15 < \s \leq 0 $ as in Proposition \ref{PROP:energy2}, 
we need to go through a limiting argument to obtain 
the identity \eqref{E8}.
This argument, however, is standard and thus we omit details.

\subsection{First few steps of normal form reductions}\label{SUBSEC:NF1}

In this section and the next section, we go over normal form reductions.
The formal computation at each step 
and the resulting multilinear forms
are essentially the same as those appearing in \cite{OST}
(modulo the slightly different frequency sets $A_j$ defined in \eqref{Cj}).
In terms of the actual estimates on the multilinear forms, however, 
we closely follow the argument in \cite{OW1}.
For readers' convenience, we  present essentially the  full details.

In this section, we go over 
 the first few steps.
Let $v$ be a smooth global solution to \eqref{4NLS8}.
With $\phi(\bar n)$ and $\G(n)$ as in~\eqref{phi1} and~\eqref{G1}, 
we have
\begin{align}
\begin{split}
\frac{d}{dt}\bigg(\frac{1}{2} \|v(t) \|_{H^s}^2\bigg)
& = - \Re i 
\sum_{n \in \Z} \sum_{\G(n)}
\jb{n}^{2s}
e^{ - i \phi(\bar n) t}   v_{n_1}(t) \cj{ v_{n_2}} (t) v_{n_3}(t) \cj{ v_n}(t)\\
&  = : \Ns^{(1)}(t)(v(t)).
\end{split}
\label{XN0}
\end{align}

	\smallskip
	
\begin{remark}\rm

(i)
Due to the presence of 
the phase factors
in their definitions, 
the multilinear forms such as $\Ns^{(1)}(t)(v(t))$ 
are non-autonomous in $t$.
In the following, however, we establish nonlinear estimates
on these multilinear forms, uniformly in  $t \in \R$,
by simply using $|e^{-i \phi (\bar n) t}| = 1$.
Hence, 
we  suppress such $t$-dependence
when there is no confusion.

\smallskip

\noi
(ii)
The complex conjugate signs on  $v_{n_j}$ do not play any significant role.
Hereafter,  we drop the complex conjugate sign.

\end{remark}

In view of \eqref{phi2} and \eqref{G1}, 
we have $|\phi(\bar n)| \geq 1 $
in \eqref{XN0}.
Then, 
by performing a normal form reduction, 
namely, 
differentiating by parts, 
and substituting the equation \eqref{4NLS8}, 
we obtain
\begin{align}
\Ns^{(1)}(v)(t)
& = 
 \Re \dt \bigg[
\sum_{\TT_1 \in \mathfrak{BT}(1)}
\sum_{\substack{ \bf n \in \mathfrak{N}(\TT_1) }}
\jb{n_r}^{2s}
\frac{  e^{- i  \phi_1 t } }{\phi_1}
\prod_{a \in \TT_1^\infty} v_{n_{a}}
\bigg]\notag \\
& \hphantom{X}
- \Re 
\sum_{\TT_1 \in \mathfrak{BT}(1)}
\sum_{\substack{ \bf n \in \mathfrak{N}(\TT_1) }}
\jb{n_r}^{2s}
\frac{ e^{- i  \phi_1 t } }{\phi_1}
\dt\bigg(\prod_{a \in \TT^\infty_1} v_{n_{a}}\bigg)
\notag\\
& =  \Re \dt \bigg[
\sum_{\TT_1 \in \mathfrak{BT}(1)}
\sum_{\substack{ \bf n \in \mathfrak{N}(\TT_1)}}
\jb{n_r}^{2s}
\frac{ e^{- i  \phi_1 t } }{\phi_1}
\prod_{a \in \TT^\infty_1} v_{n_{a}}
\bigg]\notag \\
& \hphantom{X}
-  \Re 
\sum_{\TT_1 \in \mathfrak{BT}(1)}
\sum_{b\in \TT_1^\infty}
\sum_{\substack{ \bf n \in \mathfrak{N}(\TT_1) }}
\jb{n_r}^{2s}
\frac{ e^{- i  \phi_1 t } }{\phi_1}
\RR(v)_{n_b}
\prod_{a \in \TT^\infty_1\setminus\{b\}} v_{n_{a}}\notag \\
& \hphantom{X}
- \Re 
\sum_{\TT_2 \in \mathfrak{BT}(2)}
\sum_{\substack{  \bf n \in \mathfrak{N}(\TT_2) }}
\jb{n_r}^{2s}
\frac{ e^{- i  ( \phi_1+ \phi_2) t } }{\phi_1}
\prod_{a \in \TT^\infty_2} v_{n_{a}}
\notag\\
& =: \dt \Ns_0^{(2)}(v)(t) +  \Rs^{(2)}(v)(t) + \Ns^{(2)}(v)(t).
\label{N1}
\end{align}

\noi
In the second equality, we
used the equation \eqref{4NLS8}
to replace  $\dt v_{n_b}$ by 
the resonant part $\RR(v)_{n_b}$ and the non-resonant part $\NN(v)_{n_b}$.
Note that the substitution of   $\NN(v)_{n_b}$
amounts to 
extending the tree $\TT_1\in \mathfrak{BT}(1)$
(and ${\bf n }\in \mathfrak{N}(\TT_1)$)
to $\TT_2 \in \mathfrak{BT}(2)$
(and to  ${\bf n }\in \mathfrak{N}(\TT_2)$, respectively)
by replacing  the terminal node $b \in \TT^\infty_1$
into a non-terminal node with three children $b_1, b_2,$ and $b_3$.

\begin{remark} \label{REM:tdepend} \rm
Strictly speaking, the  phase factor appearing in $\Ns^{(2)}(v)$
may be $\phi_1 - \phi_2$
when the time derivative falls on the terms with the complex conjugate.
In the following, however, we simply write it as $\phi_1 + \phi_2$ since
it does not make any difference in our analysis.
Also, we often replace $\pm 1$ and $\pm i$ by $1$
for simplicity when they do not play an important  role.
Lastly, for notational simplicity, 
we drop  the real part symbol 
on  multilinear forms
with the understanding that  all the multilinear forms
appear with 
the real part symbol.

\end{remark}

We first estimate  the boundary term $\Ns_0^{(2)}$.
In the remaining part of this section, 
we set $\s = \s(s) = s - \frac 12 - \eps$
for some small $\eps > 0$ as in \eqref{sig}.

\begin{lemma}\label{LEM:N0}	
Let $\textup{\Ns}_0^{(2)}$ be as in \eqref{N1}. Then, for $s>0$, we have	
\begin{align}
| \Ns^{(2)}_0(v) | \les \| v \|^4_{H^\s}.
\label{YN0}
\end{align}
\end{lemma}

\begin{proof}
For notational simplicity, we drop the superscript $(1)$
in the frequencies $n^{(1)} = n_r$ and $n^{(1)}_j$.
From \eqref{MU2}, we have
\begin{align}
\sup_{n\in \Z}\sum_{\G(n)}
\frac{n_{\max}^{4s-8\s}}{|\phi_1|^2}
\les\sup_{n\in \Z}\sum_{\G(n)}
\frac{1}{|(n-n_1)(n-n_3)|^2 {n_{\max}^{4-4s +8\s}}}
\les 1,
\label{N01}
\end{align}

\noi 
provided that 
 $4-4s+8\s >0 $, namely, $s>0$. Then, by Cauchy-Schwarz inequality with $ | \mathfrak{BT}(\TT_1) |=1$ and \eqref{N01}, we have
\begin{align*}
| \Ns^{(2)}_0(v) | &\les \sum_{\TT_1 \in \mathfrak{BT}(1) }\sum_{n \in \Z} \sum_{ \substack{\bn \in \mathfrak{N}(\TT_1)\\ n_r=n }}\frac{n_{\max}^{2s-4\s}
}{| \phi_1 | }\prod_{a \in \TT_1^\infty }\jb{n_a}^{\s}v_{n_a}\\
& \leq  \| v\|_{H^\s} \Bigg\{ 
\bigg( \sup_{n\in \Z}\sum_{\substack{\G(n)}}
\frac{n_{\max}^{4s-8\s}}{|\phi_1|^2}\bigg)
\cdot \bigg(\sum_{n\in \Z} \sum_{\G(n)}  \prod_{i =1}^3 \jb{n_i}^{2\s}|v_{n_i}|^2 \bigg)
\Bigg\}^\frac{1}{2}\\
 &  \les  \|v\|_{H^\s}^4.
\end{align*}

\noi
This proves \eqref{YN0}	 
\end{proof}

Proceeding in an analogous manner, 
we obtain the following estimate on $\Rs^{(2)}$.

\begin{lemma}\label{LEM:R0}
Let $\textup{\Rs}^{(2)}$ be as in \eqref{N1}. Then, for $s>\frac{1}{4}$, we have
\begin{align*}
| \textup{\Rs}^{(2)}(v) |\les \| v \|_{H^\s}^6.
\end{align*} 
\end{lemma}

\begin{proof}	
This lemma follows from  the proof of Lemma \ref{LEM:N0}
and $\l^2 \embeds \l^6$, 
once we observe that 
\begin{align*}
\sup_{n\in \Z}\sum_{\G(n)}
\frac{n_{\max}^{4s-12\s}}{|\phi_1|^2}
\les\sup_{n\in \Z}\sum_{\G(n)}
\frac{1}{|(n-n_1)(n-n_3)|^2 {n_{\max}^{4-4s+12\s}}}
\les 1,
\end{align*}

\noi
provided that  $4-4s +12\s >0 $, namely, $s> \frac 14$.
\end{proof}

As it is, 
we cannot estimate 
$\Ns^{(2)}$ in \eqref{N1}.
By dividing the frequency space into 
$A_1$ defined in~\eqref{Cj}
and its complement $A_1^c$,
we split $\Ns^{(2)}$ as 
\begin{equation} \label{N^2_1}
\Ns^{(2)} = \Ns^{(2)}_1 + \Ns_2^{(2)},
\end{equation}

\noi
where $\Ns^{(2)}_1$ is the restriction of $\Ns^{(2)}$
onto $A_1$
and
$\Ns_2^{(2)} := \Ns^{(2)} - \Ns^{(2)}_1$.
Thanks to the frequency restriction $A_1$, 
we can estimate the first term $\Ns^{(2)}_1$.

\begin{lemma}\label{LEM:N^2_1}
Let $\textup{\Ns}^{(2)}_1$ be as in \eqref{N^2_1}. Then, for $s > \frac 3{10}$, we have
\begin{align*}
| \textup{\Ns}^{(2)}_1(v) | \les \| v \|_{H^\s}^6.
\end{align*}
\end{lemma}

\begin{proof}
On $A_1$, we have 
 $|\phi_2|\les |\phi_1|$.
Then,  from  \eqref{MU2},  we have 
\begin{align} 
\begin{split}
\sup_{n \in \Z} & 
\sum_{\substack{{\bf n} \in \mathfrak{N}(\TT_2)\\ n_r = n\\|\phi_2| \les |\phi_1| }} 
 \frac{(n^{(1)}_{\max})^{4s-6\s} (n^{(2)}_{\max})^{-6\s}}{|\phi_1|^2} \\
&  \les
\sup_{n \in \Z}  \sum_{\substack{{\bf n} \in \mathfrak{N}(\TT_2)\\  n_r = n}} 
\frac{1} {|\mu_1 |^\al  |\mu_2|^{2-\al} (n^{(1)}_{\max})^{-4s+ 6\s+2\al}   (n^{(2)}_{\max} )^{6\s+4-2\al }}
\end{split}
\label{YN1}
\end{align}

\noi 
for any $0\le\al\le2$.
We  impose 
$-4s+ 6\s+2\al>  0$ and $6\s +4-2\al>  0$, 
namely, 
\begin{align}
s>-\al+\frac 32
\qquad \text{and} \qquad s>\frac \al3  -\frac 16.
\label{YN2}
\end{align}

\noi
In view of  the powers of 
$n^{(1)}_{\max}$ and $n^{(2)}_{\max}$ on the left-hand side of \eqref{YN1}, 
we may assume that $1 \le \al \le 2$.
From 
 $|\mu_2| \les (n^{(2)}_{\max})^2$,  we have
\begin{align*}
|\mu_2|^{2-\al}    (n^{(2)}_{\max} )^{6\s+4-2\al }
\ges 
|\mu_2|^{3s - 2\al + \frac 52 - 3\eps}.
\end{align*} 

\noi
Now, we impose 
$3s - 2\al + \frac 52  > 1$, namely, 
\begin{align}
s>\frac 23 \al - \frac 12.
\label{YN3}
\end{align}

\noi 
Under the conditions \eqref{YN2} and \eqref{YN3}, 
it follows  from \eqref{YN1} that there exists $\dl > 0$ such that 
\begin{align} 
\sup_{n \in \Z}
\sum_{\substack{{\bf n} \in \mathfrak{N}(\TT_2)\\ n_r = n\\|\phi_2| \les |\phi_1| }} 
\frac{(n^{(1)}_{\max})^{4s-6\s} (n^{(2)}_{\max})^{-6\s}}{|\phi_1|^2}  
&  \les
\sup_{n \in \Z}  \sum_{\substack{{\bf n} \in \mathfrak{N}(\TT_2)\\  n_r = n}} 
\frac{1} {|\mu_1 |^{1+\dl}  |\mu_2|^{1+\dl} } \les 1.
\label{TT3}
\end{align}

\noi
By optimizing 
the conditions \eqref{YN2} and \eqref{YN3}
with $\al = \frac 65$, 
we obtain the restriction $s > \frac 3{10}$.

\smallskip

\noi
$\bullet$ {\bf Case 1:}
We first consider the case  $\Pi_2(\TT_2) = \{r_2\}$.
Namely, the second root node $r_2$ is a terminal node.
By Cauchy-Schwarz inequality with \eqref{TT3}, we have
\begin{align*}
| \Ns^{(2)}_1(v) | 
& \les 
\sum_{n\in \Z}
\sum_{\substack{\TT_2 \in \mathfrak{BT}(2)\\ \Pi_2(\TT_2) = \{r_2\}}}
\sum_{\substack{{\bf n} \in \mathfrak{N}(\TT_2)\\n_r = n }} 
\ind_{A_1}
\frac{\jb{n_r}^{2s}}{|\phi_1|} \prod_{a \in \TT^\infty_2} v_{n_a} 
\notag  \\
& \les
\| v\|_{H^\s} \Bigg\{ \sum_{n\in \Z} \Bigg(
\sum_{\substack{\TT_2 \in \mathfrak{BT}(2)\\ \Pi_2(\TT_2) = \{r_2\}}}
\sum_{\substack{{\bf n} \in \mathfrak{N}(\TT_2)\\ n_r = n}} 
\ind_{A_1}
\frac{\jb{n}^{2s-\s}}{|\phi_1|} \prod_{a \in \TT^\infty_2 \setminus \{r_2\} } v_{n_a} \Bigg)^2 \Bigg\}^{\frac12} 
\notag  \\
& \les \|v\|_{H^\s} 
\sup_{\substack{\TT_2 \in \mathfrak{BT}(2)\\ \Pi_2(\TT_2) = \{r_2\}}}
\bigg( \sup_{n\in \Z} \sum_{\substack{{\bf n} \in \mathfrak{N}(\TT_2)\\ n_r =n}} 
\ind_{A_1}
\frac{(n^{(1)}_{\max})^{4s-6\s} (n^{(2)}_{\max})^{-6\s}}{|\phi_1|^2} \bigg)^\frac{1}{2} \notag \\
& \qquad  \qquad \times 
\bigg(  
\sum_{n\in \Z}
\sum_{\substack{{\bf n} \in \mathfrak{N}(\TT_2)\\n_r = n}} 
\prod_{a \in \TT^\infty_2\setminus \{r_2\}} \jb{n_a}^{2\s}|v_{n_a}|^2
\bigg)^\frac{1}{2}\notag  \\
& \les \|v\|_{H^\s}^6.
\end{align*}

\smallskip

\noi
$\bullet$ {\bf Case 2:}
Next, we  consider the case  $\Pi_2(\TT_2) \ne \{r_2\}$.
In this case, we need to modify the argument above since
the frequency $n_r = n$ does not correspond to a terminal node.
Noting that $\TT_2^\infty = \Pi_1(\TT_2)^\infty \cup \Pi_2(\TT_2)^\infty$, 
we have 
\begin{align}
\sum_{\substack{{\bf n} \in \mathfrak{N}(\TT_2)\\{  n}_r = n}} 
\prod_{a \in \TT^\infty_2} |v_{n_a}|^2
= \prod_{j = 1}^2 \bigg(\sum_{\substack{{\bf n} \in \mathfrak{N}(\Pi_j(\TT_2))\\{  n}_{r_j} = n}} 
\prod_{a_j \in \Pi_j(\TT_2)^\infty} |v_{n_{a_j}}|^2
\bigg).
\label{YN4}
\end{align}

\noi
Then, from \eqref{TT3} and \eqref{YN4}, 
we have
\begin{align*}
 |   \Ns^{(2)}_1(v)|
& \les 
\sum_{n\in \Z}
\sum_{\substack{\TT_2 \in \mathfrak{BT}(2)\\ \Pi_2(\TT_2) \ne \{r_2\}}}
\sum_{\substack{{\bf n} \in \mathfrak{N}(\TT_2)\\ n_r = n}} 
\ind_{A_1}
\frac{\jb{n_r}^{2s}}{|\phi_1|} \prod_{a \in \TT^\infty_2} v_{n_a} 
\notag  \\
& \les 
\sup_{\substack{\TT_2 \in \mathfrak{BT}(2)\\ \Pi_2(\TT_2) \ne \{r_2\}}}
\sum_{n \in \Z}
\bigg(  \sum_{\substack{{\bf n} \in \mathfrak{N}(\TT_2)}} 
\ind_{A_1}
\frac{(n^{(1)}_{\max})^{4s-6\s} (n^{(2)}_{\max})^{-6\s}}{|\phi_1|^2}\, \bigg)^\frac{1}{2}\\
& \hphantom{XXXXXXX}
\times 
\bigg( \sum_{\substack{{\bf n} \in \mathfrak{N}(\TT_2)\\ n_r = n}}
\prod_{a \in \TT^\infty_2} \jb{n_a}^{2\s}|v_{n_a}|^2
\bigg)^\frac{1}{2}\notag  \\
& \les 
\sup_{\substack{\TT_2 \in \mathfrak{BT}(2)\\ \Pi_2(\TT_2) \ne \{r_2\}}}
\sum_{n \in \Z}
\bigg( \sum_{\substack{{\bf n} \in \mathfrak{N}(\TT_2)\\n_r = n}}
\prod_{a \in \TT^\infty_2} \jb{n_a}^{2\s}|v_{n_a}|^2
\bigg)^\frac{1}{2}\notag  \\
& \les 
\sup_{\substack{\TT_2 \in \mathfrak{BT}(2)\\ \Pi_2(\TT_2) \ne \{r_2\}}}
\sum_{n\in \Z} \, \prod_{j = 1}^2  \bigg( 
\sum_{\substack{{\bf n} \in \mathfrak{N}(\Pi_j(\TT_2))\\ n_{r_j} = n}} 
\prod_{a_j \in \Pi_j(\TT_2)^\infty} \jb{n_{a_j}}^{2\s}|v_{n_{a_j}}|^2 \bigg) ^\frac{1}{2}\notag \\
& \les 
\sup_{\substack{\TT_2 \in \mathfrak{BT}(2)\\ \Pi_2(\TT_2) \ne \{r_2\}}}
\prod_{j = 1}^2   \bigg( 
\sum_{n\in \Z}  \sum_{\substack{{\bf n} \in \mathfrak{N}(\Pi_j(\TT_2))\\n_{r_j} = n}} 
\prod_{a_j \in \Pi_j(\TT_2)^\infty} \jb{n_{a_j}}^{2\s}|v_{n_{a_j}}|^2 \bigg) ^\frac{1}{2}\notag \\
& \les 
\|v\|_{H^\s}^6.
\end{align*}

\noi
This completes the proof of Lemma \ref{LEM:N^2_1}.
\end{proof}

Before moving onto the next subsection, 
let us briefly describe 
how to handle 
the highly non-resonant part
  $\Ns^{(2)}_2$ in \eqref{N^2_1}.
On the support of $\Ns^{(2)}_2$, i.e.~on $A_1^c$, we have 
\begin{equation} 
|\phi_1 + \phi_2|  \gg 6^3 |\phi_1|
\label{C1^c}
\end{equation}

\noi
Namely, the phase function $\phi_1 + \phi_2$ is ``large'' in this case
and hence we can exploit this fast oscillation
by applying the second step of the normal form reduction:
\begin{align*}
\Ns^{(2)}_2 (v)
& = \dt \bigg[\sum_{\TT_2 \in \mathfrak{BT}(2)}
\sum_{{\bf n} \in \mathfrak{N}(\TT_2)}
\ind_{A_1^c} \frac{\jb{n_r}^{2s}e^{- i( \phi_1 + \phi_2 )t } }{\phi_1(\phi_1+\phi_2)}
\prod_{a \in \TT^\infty_2} v_{n_{a}} \bigg]\notag \\
& \hphantom{X} -  
\sum_{\TT_2 \in \mathfrak{BT}(2)}
\sum_{b \in\TT^\infty_2} 
\sum_{{\bf n} \in \mathfrak{N}(\TT_2)}
\ind_{A_1^c}\frac{\jb{n_r}^{2s}e^{- i( \phi_1 + \phi_2)t } }{\phi_1(\phi_1+\phi_2)}
\, \RR(v)_{n_b}
\prod_{a \in \TT^\infty_2 \setminus \{b\}} v_{n_{a}} \notag \\
& \hphantom{X} 
- \sum_{\TT_3 \in \mathfrak{BT}(3)}\sum_{{\bf n} \in \mathfrak{N}(\TT_3)}
\ind_{A_1^c}\frac{\jb{n_r}^{2s}e^{- i( \phi_1 + \phi_2 +\phi_3)t } }{\phi_1(\phi_1+\phi_2)}
\,\prod_{a \in \TT^\infty_3} v_{n_{a}} \notag\\
& =: \dt \Ns^{(3)}_0(v) + \Rs^{(3)}(v) + \Ns^{(3)}(v).
\end{align*}

\noi
Using \eqref{C1^c}, 
we can estimate 
the first two terms  $\Ns^{(3)}_0$ and $ \Rs^{(3)}$ on the right-hand side 
 in a straightforward manner.
See Lemmas \ref{LEM:N^J+1_0} and \ref{LEM:N^J+1_r} below.
As for the last term $\Ns^{(3)}$,  
we split it as 
$\Ns^{(3)} = \Ns^{(3)}_1 + \Ns^{(3)}_2$, 
where $\Ns^{(3)}_1$ and  $\Ns^{(3)}_2$
are  the restrictions
onto $A_2$ and its complement $A_2^c$, respectively.
By exploiting the frequency restriction on $A_1^c \cap A_2$,
we can  estimate the first term  $\Ns^{(3)}_1$ (see Lemma \ref{LEM:N^J+1_1} below).
As for the second term $\Ns^{(3)}_2$, 
we apply the third step of the normal form reductions.
In this way, we iterate  normal form reductions
in an indefinite manner.

\subsection{General step}\label{SUBSEC:NF2}

After the $J$th step, we have
\begin{align} \label{N^J+1}
\Ns^{(J)}_2 (v)
& = \dt \bigg[ 
\sum_{\TT_J \in \mathfrak{BT}(J)}
\sum_{{\bf n} \in \mathfrak{N}(\TT_J)}
\ind_{\bigcap_{j = 1}^{J-1} A_j^c}
\frac{\jb{n_r}^{2s}e^{- i \wt{\phi}_Jt } }{\prod_{j = 1}^J \wt{\phi}_j}
\, \prod_{a \in \TT^\infty_J} v_{n_{a}}
\bigg]\notag \\
& \hphantom{X} 
- \sum_{\TT_{J} \in \mathfrak{BT}(J)}
\sum_{b \in\TT^\infty_J} 
\sum_{{\bf n} \in \mathfrak{N}(\TT_J)}
\ind_{\bigcap_{j = 1}^{J-1} A_j^c}
\frac{\jb{n_r}^{2s}e^{- i \wt{\phi}_Jt } }{\prod_{j = 1}^J \wt{\phi}_j}
\, \RR(v)_{n_b}
\prod_{a \in \TT^\infty_J \setminus \{b\}} v_{n_{a}}  \notag \\
& \hphantom{X} 
- \sum_{\TT_{J+1} \in \mathfrak{BT}(J+1)}
\sum_{{\bf n} \in \mathfrak{N}(\TT_{J+1})}
\ind_{\bigcap_{j = 1}^{J-1} A_j^c}
\frac{\jb{n_r}^{2s}e^{- i \wt{\phi}_{J+1}t } }{\prod_{j = 1}^J \wt{\phi}_j}
\,\prod_{a \in \TT^\infty_{J+1}} v_{n_{a}} \notag\\
& =: \dt \Ns^{(J+1)}_0 (v)+ \Rs^{(J+1)}(v) + \Ns^{(J+1)}(v).
\end{align}

\noi
On $\bigcap_{j = 1}^{J-1} A_j^c$, we have 
 $|\phi_1|\geq 1$ and 
\begin{equation} \label{muj}
|\wt{\phi}_j|  \gg 
(2j+2)^{3}\max \big( |\wt{\phi}_{j-1}|, 
|\phi_1|\big) \geq
(2j+2)^{3}
\end{equation}

\noi
for $j = 2, \dots, J.$
As in \cite{GKO, OST, OW1}, 
we control
  the rapidly growing
cardinality
$c_J = |\mathfrak{BT}(J)| $  defined in \eqref{cj1}
by the growing
constant $(2j+2)^3$ appearing in \eqref{muj}.

First, we estimate $\Ns^{(J+1)}_0$ and  $\Rs^{(J+1)}$.

\begin{lemma}\label{LEM:N^J+1_0}	
Let $\textup{\Ns}_0^{(J+1)}$ be as in \eqref{N^J+1}. Then, for any $s > \frac{1}{6}$, 
we have 
\begin{align} 
| \Ns_0^{(J+1)}(v)|  \les 
\frac{1}{\prod_{j = 2}^J(2j+2)^{\frac 13 }}
\|v\|_{H^\s}^{2J+2}. 
\label{YN5}
\end{align}

\noi	
 Here, the implicit constant is independent of $J$. 
\end{lemma}

\begin{proof}
From \eqref{Cj2}, we have 
\begin{align*} 
|\phi_j| \lesssim \max\big(|\wt{\phi}_{j-1}|, |\wt{\phi}_j|\big). 
\end{align*}

\noi
Then, in view of \eqref{muj}, we have 
\begin{align} 
\label{mujj2}
(2j)^3  |\phi_j| \ll |\wt{\phi}_{j-1}| |\wt{\phi}_j|.
\end{align}

\noi
Hence,  from \eqref{mujj2} and then  \eqref{muj},  we have
\begin{align} 
\label{mujj3}
\prod_{j=1}^{J} |\wt\phi_j|^2 
& \gg 
|\phi_1| |\wt \phi_J|
\prod_{j=2}^{J}\Big( (2j)^3  |\phi_j|\Big) 
\gg 
| \phi_1 |^2 \prod_{j=2}^{J}\Big( (2j+2)^3  |\phi_j|\Big).
\end{align}

We only discuss the case  $\Pi_2(\TT_J) = \{r_2\}$
since the modification is straightforward if $\Pi_2(\TT_J) \ne \{r_2\}$.
Given $s > \frac 16$, 
there exists small $\dl > 0$ such that 
\begin{align}
\frac{(n^{(j)}_{\max})^{-6\s}}{|\phi_j|} 
\sim \frac{(n^{(j)}_{\max})^{-6\s}}{|\mu_j| (n^{(j)}_{\max})^{2}} \les \frac{1}{|\mu_j|^{1+\dl} }   .
\label{mujj31}
\end{align}

\noi
Similarly, we  have 
\begin{align}
\frac{(n^{(1)}_{\max})^{4s-6\s}}{|\phi_1|^2} 
\sim \frac{(n^{(1)}_{\max})^{4s-6\s}}{|\mu_1|^2 (n^{(1)}_{\max})^{4}} \les \frac{1}{|\mu_1|^{2} }. 
\label{mujj32}  
\end{align}

\noi 
Then, 
from \eqref{mujj3}, \eqref{mujj31}, and \eqref{mujj32}, we have
\begin{align}
\sup_{n\in \Z}
& \sum_{\substack{{\bf n} \in \Nf(\TT_J)\\ n_r = n}}
\ind_{\bigcap_{j = 1}^{J-1} A_j^c}\cdot 
(n^{(1)}_{\max})^{4s}
\prod_{j = 1}^J \frac{  (n^{(j)}_{\max})^{-6\s}}{|\wt{\phi}_j|^2} \notag\\ 
& \ll \frac{1 }{\prod_{j = 1}^J(2j+2)^{3}}
\cdot  \sup_{n\in \Z}
\sum_{\substack{{\bf n} \in \Nf(\TT_J)\\ n_r = n\\ \phi_j \neq 0\\ j = 1, \dots, J} } 
\ind_{\bigcap_{j = 1}^{J-1} A_j^c}
\frac{ (n^{(1)}_{\max})^{4s-6\s}   }{| \phi_1 |^2 }
\prod_{j = 2}^J  \frac{(n^{(j)}_{\max})^{-6s}}{|\phi_j|} \notag \\
& \les \frac{1 } {\prod_{j = 1}^J(2j+2)^{3}}
\cdot  \sup_{n\in \Z}
\sum_{\substack{{\bf n} \in \Nf(\TT_J)\\ n_r = n\\ \mu_j \neq 0\\ j = 1, \dots, J} } 
\frac{1}{|\mu_1 |^{2}}\prod_{j = 2}^J \frac{1}{|\mu_j|^{1+\dl}} \notag \\
& \le \frac{C^J }{\prod_{j = 1}^J(2j+2)^{3}}.
\label{mujj4}
\end{align}

\noi 
Hence, by Cauchy-Schwarz inequality and 
\eqref{mujj4}, we have
\begin{align}
| \Ns^{(J+1)}_0(v)|
& \les  \|v\|_{H^\s}
\sum_{\substack{\TT_J \in \mathfrak{BT}(J)\\\Pi_2(\TT_J) = \{r_2\}}}
\Bigg\{ \sum_{n \in \Z} 
\bigg(
\sum_{\substack{{\bf n} \in \mathfrak{N}(\TT_J)\\{ n}_r = n}}
\ind_{\bigcap_{j = 1}^{J-1} A_j^c}
\frac{\jb{n^{(1)}_{\max} }^{4s-6\s}}{|\phi_1|^{2}} 
\prod_{j = 2}^J \frac{ \jb{n^{(j)}_{\max} }^{-6\s}  }{|\wt{\phi}_j|^2}
\bigg) \notag \\
& \hphantom{XXXXXX} \times 
\bigg(\sum_{\substack{{\bf n} \in \mathfrak{N}(\TT_J)\\{ n}_r = n}} 
\prod_{a \in \TT^\infty_J\setminus\{r_2\}} |v_{n_a}|^2
\bigg)\Bigg\}^\frac{1}{2}  \notag \\
& \les \frac{c_J\cdot C^\frac J2}{\prod_{j = 2}^J(2j+2)^{\frac{3}{2}}}
\|v\|_{H^\s}^{2J+2}.
\label{YN5a}
\end{align}

\noi
Then, the desired bound \eqref{YN5} follows from \eqref{cj1}.
\end{proof}    

\begin{remark}\rm
At the first inequality in \eqref{YN5a}, we needed the full power 
$\jb{n^{(j)}_{\max} }^{-6\s}$ only for those $j$'s 
such that the three terminal nodes of 
the tree  added in the $(j-1)$th step are also in $\TT_J^\infty$.
For example, $j = J$ satisfies this condition.
For other values of $j$, a smaller power may suffice.
Note, however, that 
we need to use 
\eqref{mujj31} at least for $j = J$, thus requiring the regularity
restriction $s > \frac 16$.
We therefore simply used the maximum power
$\jb{n^{(j)}_{\max} }^{-6\s}$ for all $j = 1, \dots, J$
at the first inequality in \eqref{YN5a}.
The same comments applies to Lemmas \ref{LEM:N^J+1_r}
and \ref{LEM:N^J+1_1}.
\end{remark}

\begin{lemma}
\label{LEM:N^J+1_r}
Let $\textup{\Rs}^{(J+1)}$ be as in \eqref{N^J+1}. Then, for any $ \frac{1}{6} < s \le \frac 12$, 
we have 
\begin{align} 
|\Rs^{(J+1)}(v)|  \les 
\frac{1}{\prod_{j = 2}^J(2j+2)^{\frac 13 }}
\|v\|_{H^\s}^{2J+4}.
\label{YN6}
\end{align}

\noi
Here, the implicit constant is independent of $J$.
\end{lemma}

\begin{proof}
We consider two cases: 
(i)  $| \wt{\phi}_J | \ges | \phi_J | $ and (ii) $| \wt{\phi}_J| \ll | \phi_J | $.

\smallskip

\noi
$\bullet$ {\bf Case 1:}
 $| \wt{\phi}_J | \ges | \phi_J | $. 
\quad 
From \eqref{MU2}, we have 
\begin{align}
(  n_{\max}^{(J)} )^{-4\s}\les | \phi_J |^{-2\s}
\label{R1}
\end{align}

\noi 
for $\s \le 0$, namely, $ s\le \frac 12$.
From \eqref{muj}, we  have 
\begin{align}  
| \wt{\phi}_J |=| \wt{\phi}_J |^{-2\s} | \wt{\phi}_J |^{1+2\s}
\gg | \wt{\phi}_J |^{-2\s} | \phi_1 |^{1+2\s}\ges | \phi_J|^{-2\s}| \phi_1 |^{1+2\s},
\label{R2}
\end{align} 

\noi 
provided that 
 $1+2\s\ge 0$, namely,  $s>0$.
We also observe that
\begin{align} 
\frac{ ( n_{\max}^{(1)} )^{4s-6\s}}{ | \phi_1 |^{2+2\s }}
\sim \frac{1}{ | \mu_1 |^{2+2\s}( n_{\max}^{(1)} )^{4-4s +10\s} }.
\label{R3} 
\end{align}

\noi 
Note that 
$4 - 4s +10\s>0$ and $2+2\s>1$, provided that $s>\frac 16$. 
Then,  by applying \eqref{mujj3} and \eqref{R2}
followed by  \eqref{mujj31}, \eqref{R1},  and \eqref{R3}, we have
\begin{align}
\sup_{n\in \Z}& 
\sum_{\substack{{\bf n} \in \Nf(\TT_J)\\ n_r = n}}
\ind_{\bigcap_{j = 1}^{J-1} A_j^c}\cdot 
(n_{\max}^{(J)})^{-4\s}
(n^{(1)}_{\max})^{4s}
\prod_{j = 1}^J \frac{  (n^{(j)}_{\max})^{-6\s}}{|\wt{\phi}_j|^2}\notag\\ 
& \ll \frac{1 }{\prod_{j = 2}^J(2j)^{3}}
\cdot  \sup_{n\in \Z}
\sum_{\substack{{\bf n} \in \Nf(\TT_J)\\ n_r = n\\ \phi_j \neq 0\\ j = 1, \dots, J} } 
\ind_{\bigcap_{j = 1}^{J-1} A_j^c} 
\frac{ (n_{\max}^{(J)} )^{-4\s} }{| \wt{\phi}_J | } \frac{(n_{\max}^{(1)})^{4s-6\s} }{| \phi_1 |} \prod_{j = 2}^J  \frac{(n^{(j)}_{\max})^{-6\s}}{|\phi_j|} \notag \\
& \ll \frac{1 }{\prod_{j = 2}^J(2j)^{3}}
\cdot  \sup_{n\in \Z}
\sum_{\substack{{\bf n} \in \Nf(\TT_J)\\ n_r = n\\ \phi_j \neq 0\\ j = 1, \dots, J} } 
\frac{ (n_{\max}^{(J)} )^{-4\s} }{| \phi_J |^{-2\s} } \frac{(n_{\max}^{(1)})^{4s-6\s} }{| \phi_1 |^{2+2\s}} \prod_{j = 2}^J  \frac{(n^{(j)}_{\max})^{-6\s}}{|\phi_j|} \notag \\
& \les \frac{C^J }{\prod_{j = 2}^J(2j)^{3}}.
\label{R4}
\end{align}

\noi 
Hence,   proceeding as in 
\eqref{YN5a}
with \eqref{R4} and \eqref{cj1}, 
we obtain 
\eqref{YN6} in this case.

\smallskip

\noi
$\bullet$ {\bf Case 2:}
 $| \wt{\phi}_J | \ll | \phi_J |$. 
 \quad In this case, we have $| \phi_J | \sim |  \wt{\phi}_{J-1} | $. 
 From \eqref{mujj2}, we have
\begin{align}
\prod_{j=1}^J| \wt{\phi}_j |^2 
\gg      | \phi_1 | | \wt{\phi}_{J-1} | | \wt{\phi}_J |^2 
 \prod_{j=2}^{J-1} \Big( (2j)^3 | \phi_j | \Big).
\label{YN7}
\end{align}

\noi
From \eqref{muj}, we also have 
\begin{align}
\begin{split}
|\wt \phi_J| & \gg (2J+2)^3 |\wt \phi_{J-1}| \sim (2J+2)^3 |\phi_J|,\\
| \wt{\phi}_{J-1} | & \gg (2J)^3 | \phi_1 | .
\end{split}
\label{YN8} 
\end{align}

\noi
Thus, from \eqref{YN7} and \eqref{YN8}, we have 
\begin{align} 
\prod_{j=1}^J | \wt{\phi}_j |^2 
\gg  | \phi_1 |  | \phi_J | \prod_{j=1}^J \Big( (2j+2)^3 | \phi_j |\Big).
\label{R5}
\end{align}

Note that 
\begin{align}
\frac{ ( n_{\max}^{(J)} )^{-10\s }    }{ | \phi_{J}|^{2} }\sim \frac{( n_{\max}^{(J)} ) ^{-10\s} }{ | \mu_{J} |^{2} ( n_{\max}^{(J)} )^{4} }\le\frac{1 }{ | \mu_{J} |^{2}}, 
\label{R6}
\end{align}

\noi 
provided that  $4+10\s>0$, namely, 
 $s>\frac {1}{10}$. 
Then, from 
\eqref{R5}, \eqref{mujj31}, \eqref{mujj32}, and \eqref{R6}, we have 
\begin{align}
\sup_{n\in \Z} & 
\sum_{\substack{{\bf n} \in \Nf(\TT_J)\\ n_r = n}}
\ind_{\bigcap_{j = 1}^{J-1} A_j^c}\cdot 
(n_{\max}^{(J)})^{-4\s}
(n^{(1)}_{\max})^{4s}
\prod_{j = 1}^J \frac{  (n^{(j)}_{\max})^{-6\s}}{|\wt{\phi}_j|^2}\notag\\ 
& \ll \frac{1 }{\prod_{j = 1}^J(2j+2)^{3}}
\cdot  \sup_{n\in \Z}
\sum_{\substack{{\bf n} \in \Nf(\TT_J)\\ n_r = n\\ \phi_j \neq 0\\ j = 1, \dots, J} } 
\frac{ (n_{\max}^{(J)} )^{-10\s} }{| \phi_J |^2 } \frac{(n_{\max}^{(1)})^{4s-6\s} }{| \phi_1 |^2} \prod_{j = 2}^{J-1}  \frac{(n^{(j)}_{\max})^{-6\s}}{|\phi_j|} \notag \\
& \les \frac{C^J }{\prod_{j = 1}^J(2j+2)^{3}}.
\label{R7}
\end{align}

\noi 
Hence, 
   proceeding as in 
\eqref{YN5a}
with \eqref{R7} and \eqref{cj1}, 
we obtain 
\eqref{YN6} in this case.
\end{proof}

Finally, we consider  $\Ns^{(J+1)}$.
As before, we write
\begin{equation} \label{N^J+1_1}
\Ns^{(J+1)} = \Ns^{(J+1)}_1 + \Ns^{(J+1)}_2,
\end{equation}

\noi
where $\Ns^{(J+1)}_1$ is the restriction of $\Ns^{(J+1)}$ onto $A_J$ defined in \eqref{Cj}
and
$\Ns^{(J+1)}_2 := \Ns^{(J+1)} - \Ns^{(J+1)}_1$.
In the following lemma, we estimate the first term $\Ns^{(J+1)}_1$.
Then,  we apply a normal form reduction 
once again to the second term $\Ns^{(J+1)}_2$ 
as in \eqref{N^J+1}
and repeat this process indefinitely.
Lemma \ref{LEM:N^J+1_2} below shows 
that, for a smooth function $v$,  this error term 
$\Ns^{(J+1)}_2$ tends to 0 as $J \to \infty$.

\begin{lemma}
\label{LEM:N^J+1_1}
Let $\textup{\Ns}_1^{(J+1)}$ be as in \eqref{N^J+1}. Then, for any $s > \frac 3{10}$, 
we have 
\begin{align} 
| \textup{\Ns}_1^{(J+1)}(v)|  \les 
\frac{1 }{\prod_{j = 2}^J(2j+2)^{\frac 13 }}
\|v\|_{H^\s}^{2J+4}. 
\label{YN8a}
\end{align}

\noi 
Here, the implicit constant is independent of $J$.
\end{lemma}

\begin{proof}	
On $A_J \cap A_{J-1}^c$, we have 
 $| \wt{\phi}_{J+1} | \les (2J+4)^3 | \wt{\phi}_J   | $ 
 and thus  
 \begin{align}
  | \phi_{J+1} | \les | \wt{\phi}_{J+1}|+| \wt{\phi}_J | \les J^3 | \wt{\phi}_J | .
  \label{YN9}
 \end{align}

 \noi
Then, from \eqref{mujj2}, \eqref{YN9}, 
and  \eqref{muj}, 
we have
\begin{align}
\begin{split}
J^3 \prod_{j=1}^{J} | \wt{\phi}_j |^2 
& \gg  
| \phi_1 |
(  J^3 | \wt{\phi}_J   | )^{1-\al}
(J^3 | \wt{\phi}_J |  )^{\al}
 \prod_{j=2}^{J}\Big((2j)^3 | \phi_j |\Big)\\
& \ges  | \phi_1 |^{2-\al} | \phi_{J+1} |^{\al}
 \prod_{j=2}^{J}\Big((2j)^3 | \phi_j |\Big)
\end{split}
\label{NJ3}
\end{align} 

\noi 
for  $0\le\al \le1$.

Writing
\begin{align} 
\frac{ ( n_{\max}^{(1)})^{4s -6\s}    }{ | \phi_1|^{2-\al } }
\sim \frac{1 }{| \mu_1 |^{2-\al} (n_{\max}^{(1)})^{-4s + 6\s - 2\al + 4} }
\label{NJ4}
\end{align}

\noi and
\begin{align}
\frac{(n_{\max}^{(J+1)})^{-6\s }    }{ | \phi_{J+1}|^{\al} }
\sim \frac{1}{| \mu_{J+1} |^{\al} ( n_{\max}^{(J+1)})^{6\s+2\al} }, 
\label{NJ5}
\end{align} 

\noi
we impose 
$-4s + 6\s - 2\al + 4 > 0$ and  $6\s+2\al>0$, namely, 
\begin{align}
s>\al -\frac12 
\qquad  \text{and}\qquad 
 s>-\frac{\alpha}{3} + \frac12 . 
\label{YN10}
\end{align}

\noi
From 
 $|\mu_{J+1}| \les (n^{(J+1)}_{\max})^2$,  we have
\begin{align*}
| \mu_{J+1} |^{\al} ( n_{\max}^{(J+1)})^{6\s+2\al}
\ges 
|\mu_2|^{3s + 2\al - \frac 32 - 3\eps}.
\end{align*} 

\noi
We now impose 
$3s + 2\al - \frac 32  > 1$, 
 namely, 
\begin{align}
s>- \frac 23 \al + \frac 56.
\label{YN11}
\end{align}

\noi 
\noi
By optimizing 
the conditions \eqref{YN10} and \eqref{YN11}
with $\al = \frac 45$, 
we obtain the restriction $s > \frac 3{10}$.
Hence,  for $s > \frac 3{10}$, 
it follows from  \eqref{NJ3}, \eqref{NJ4}, \eqref{NJ5}, and \eqref{mujj31}
that 
\begin{align}
\sup_{n\in \Z}
& 
\sum_{\substack{{\bf n} \in \Nf(\TT_{J+1})\\ n_r = n}}
\ind_{A_J \cap (\bigcap_{j = 1}^{J-1} A_j^c)} \cdot 
(n^{(J+1)}_{\max})^{-6\s} 
(n^{(1)}_{\max})^{4s}
\prod_{j = 1}^J \frac{   (n^{(j)}_{\max})^{-6\s}}{|\wt{\phi}_j|^2} \notag\\
& \ll \frac{ J^3 } {\prod_{j = 2}^{J} (2j)^{3}}
\cdot  \sup_{n\in \Z}
\sum_{\substack{{\bf n} \in \Nf(\TT_{J+1})\\ n_r = n\\|\phi_j| \neq 0 
		\\ j = 1, \dots, J+1} } 
 \frac{ (n_{\max}^{(J+1)})^{-6\s} }{ | \phi_{J+1} |^\al  } 
 \frac{(n_{\max}^{(1)} )^{4s-6\s} }{| \phi_1 |^{2-\al}} 
 \prod_{j = 2}^{J} \frac{(n^{(j)}_{\max})^{-6\s}}{|\phi_j|} \notag \\
&\les \frac{J^3}{\prod_{j = 2}^{J} (2j)^{3}}
\cdot 
\sup_{n\in Z}\sum_{\substack{ \bf n \in \Nf(\TT_{J+1}) \\ n_r=n   } }  \prod_{j=1}^{J+1}\frac{1}{| \mu_j |^{1+\dl}} \notag  \\
& \les  \frac{C^{J+1} J^3 }{\prod_{j = 2}^{J} (2j)^{3}}
\label{mujj6}
\end{align}

\noi
for some small $\dl > 0$.
Then, the desired bound \eqref{YN8a} follows 
 from the Cauchy-Schwarz argument with \eqref{mujj6}. 
\end{proof}

We conclude this subsection by showing 
that the error term  $\Ns_2^{(J+1)}$
in \eqref{N^J+1_1}
tends to 0 as $J \to \infty$ under some regularity assumption on $v$.
From \eqref{N^J+1}, we have
\begin{align} \label{Er}
\Ns_2^{(J+1)}(v)
&  =- \sum_{\TT_{J+1} \in \mathfrak{BT}(J+1)}
\sum_{{\bf n} \in \mathfrak{N}(\TT_{J+1})}
\ind_{\bigcap_{j = 1}^{J} A_j^c}
\frac{\jb{n_r}^{2s}e^{- i \wt{\phi}_{J+1}t } }{\prod_{j = 1}^J \wt{\phi}_j}
\,\prod_{a \in \TT^\infty_{J+1}} v_{n_{a}}.
\end{align}

\begin{lemma}\label{LEM:N^J+1_2}
Let $\s > \frac 12$.
Then, given any  $v \in H^\s(\T)$, 
we have
\begin{align*} 
| \textup{\Ns}^{(J+1)}_2(v)|  \too 0, 
\end{align*}

	\noi
	as $J \to \infty$.
\end{lemma}

\begin{proof}
By the algebra property of $H^s(\T)$, $s > \frac 12$, 
we can easily bound \eqref{Er}
by $o_{J \to \infty} (1)\|v\|_{H^s}^{2J+4}$, 
where the decay in $J$ comes from 
\eqref{muj} for $j = 2, \dots J + 1$.
See also \cite[Subsection 4.5]{OST}.
\end{proof}

\begin{remark}\rm
We point out that one can actually prove Lemma \ref{LEM:N^J+1_2}
under a weaker regularity assumption $\s \geq \frac 16$.
See \cite[Lemma 8.15]{OW1}.
\end{remark}

\subsection{Proof of Proposition \ref{PROP:energy2}} \label{SUBSEC:energy2}
We briefly discuss the proof of  Proposition \ref{PROP:energy2}.
Let $v$ be a smooth global solution to \eqref{4NLS8}.
Then, by applying the normal form reduction $J$ times, we obtain\footnote{Once again, we are replacing $\pm 1$
	and  $\pm i$
	by 1 for simplicity since they play no role in our analysis.} 
\begin{align*} 
\frac {d}{dt} \bigg(\frac 12 \| v (t) \|_{H^s}^2\bigg)
& =  \frac {d}{dt}\bigg( \sum_{j = 2}^{J+1} \textup{\Ns}^{(j)}_{0}(v)(t)\bigg)
+ \sum_{j = 2}^{J+1} \textup{\Ns}^{(j)}_{1}(v)(t)\\
& \quad 
 + \sum_{j = 2}^{J+1} \textup{\Rs}^{(j)}(v)(t)+\Ns_{2}^{(J+1)}(v)(t). 
\end{align*}

\noi
For a smooth solution $v$, 
 Lemma \ref{LEM:N^J+1_2}
 allows us to take a  limit as $J \to \infty$, yielding
\begin{align*} 
\frac {d}{dt} \bigg(\frac 12 \| v (t) \|_{H^s}^2\bigg)
=  \frac {d}{dt}\bigg( \sum_{j = 2}^\infty \textup{\Ns}^{(j)}_{0}(v)(t)\bigg)
+ \sum_{j = 2}^\infty \textup{\Ns}^{(j)}_{1}(v)(t) + \sum_{j = 2}^\infty \textup{\Rs}^{(j)}(v)(t). 
\end{align*}

\noi
Therefore, we obtain \eqref{E8} for a smooth solution $v$ to \eqref{4NLS8}.
For a rough solution
 $v \in C(\R; H^\s(\T))$,  $- \frac 15 < \s \leq 0 $, 
 we can obtain the identity \eqref{E8} by a limiting argument.
 This argument is standard and thus we omit details.
 See, for example,  Subsection 8.5 in \cite{OW1}.
 
The bounds \eqref{E9}, \eqref{E10}, and \eqref{E11} follow
from 
Lemmas \ref{LEM:N0}, 
\ref{LEM:R0}, 
\ref{LEM:N^2_1}, 
\ref{LEM:N^J+1_0}, \ref{LEM:N^J+1_r}, and \ref{LEM:N^J+1_1}.
This proves  Proposition \ref{PROP:energy2}.

\begin{ackno}\rm
T.O.~was supported by the European Research Council (grant no.~864138 ``SingStochDispDyn"). 
K.S.~was partially supported by National Research Foundation of
Korea (grant NRF-2019R1A5A1028324).
K.S.~would like to express his gratitude to the School of Mathematics at the University of Edinburgh for its 
hospitality during his visit, 
 where this manuscript was prepared. 
 The authors would like to thank the anonymous referee for the helpful comments
  which improved the presentation  of the paper.
\end{ackno}


\begin{thebibliography}{99}





\bibitem{BIT} A.~Babin, A.~Ilyin, E.~Titi, {\it On the regularization mechanism for the periodic Korteweg-de Vries equation},
Comm. Pure Appl. Math. 64 (2011), no. 5, 591--648.


\bibitem{BOP2}
\'A.~B\'enyi, T.~Oh, O.~Pocovnicu, 
{\it On the probabilistic Cauchy theory of the cubic nonlinear Schr\"odinger equation on $\R^d$, 
$d\ge 3$},  Trans. Amer. Math. Soc. Ser. B 2 (2015), 1--50.

\bibitem{BO1} J.~Bourgain, {\it Fourier transform restriction phenomena for certain lattice subsets and applications to
nonlinear evolution equations I. Schr\"odinger equations},
Geom. Funct. Anal. 3 (1993), no. 2, 107--156.


\bibitem{BO96}
J.~Bourgain, 
{\it Invariant measures for the 2D-defocusing nonlinear Schr\"odinger equation}, 
Comm. Math. Phys. 176 (1996), no. 2, 421--445. 



\bibitem{CM}
R.~Cameron, W.~Martin, 
{\it Transformations of Wiener integrals under translations},  Ann. of Math.  45 (1944). 386--396.



\bibitem{Christ}
M.~Christ, 
{\it Power series solution of a nonlinear Schr\"odinger equation}, Mathematical aspects of nonlinear dispersive equations, 131--155, Ann. of Math. Stud., 163, Princeton Univ. Press, Princeton, NJ, 2007.

\bibitem{CKSTT2003}
J.~Colliander, M.~Keel, G.~Staffilani, H.~Takaoka, T.~Tao, 
{\it  Sharp global well-posedness for KdV and modified KdV on $\R$ and $\T$}, J. Amer. Math. Soc. 16 (2003), no. 3, 705--749.





\bibitem{CO2012}
J.~Colliander, T.~Oh, {\it Almost sure well-posedness of the periodic cubic nonlinear Schr\"odinger equation
	below $L^2(\T)$},   
Duke Math. J. 161 (2012), no. 3, 367--414. 


\bibitem{Cru1}
A.B.~Cruzeiro, {\it \'Equations diff\'erentielles ordinaires: non explosion et mesures quasi-invariantes,}
(French) 
J. Funct. Anal. 54 (1983), no. 2, 193--205.


\bibitem{Cru2}
A.B.~Cruzeiro, {\it \'Equations diff\'erentielles sur l'espace de Wiener et formules de Cameron-Martin non-lin\'eaires}, 
(French) 
J. Funct. Anal. 54 (1983), no. 2, 206--227. 


\bibitem{DT2020} 
A.~Debussche, Y.~Tsutsumi,
{\it  Quasi-invariance of Gaussian measures transported by the
cubic NLS with third-order dispersion on $\mathbf T$}, 
 J. Funct. Anal. 281 (2021), no. 3, 109032, 23 pp.



\bibitem{FO}
J.~Forlano, T.~Oh,
{\it  Normal form approach to the one-dimensional cubic nonlinear Schr\"odinger equation in Fourier-amalgam spaces}, preprint.

\bibitem{FT2019}
J.~Forlano, W.~Trenberth,
{\it On the transport of Gaussian measures under the one-dimensional fractional nonlinear Schr\"odinger equation}, 
Ann. Inst. Henri Poincar\' e, Anal. Non Lin\' eaire, 36 (2019), 1987--2025.


\bibitem{GTV}
J.~Ginibre, Y.~Tsutsumi, G.~Velo, 
{\it On the Cauchy problem for the Zakharov system,} J. Funct. Anal. 151 (1997), no. 2, 384--436. 





\bibitem{GH}
A.~Gr\"unrock, S.~Herr, 
{\it Low regularity local well-posedness of the derivative nonlinear Schr\"odinger equation with periodic initial data}, SIAM J. Math. Anal. 39 (2008), no. 6, 1890--1920.

\bibitem{GOTW}
T.~Gunaratnam, T.~Oh, N.~Tzvetkov, H.~Weber,
{\it  Quasi-invariant Gaussian measures for the nonlinear wave equation in three dimensions},
 to appear in Probab. Math. Phys.


\bibitem{GKO}
Z.~Guo, S.~Kwon, T.~Oh, 
{\it  Poincar\'e-Dulac normal form reduction for unconditional well-posedness of the periodic cubic NLS,} Comm. Math. Phys. 322 (2013), no.1, 19--48. 


\bibitem{GO}
Z.~Guo, T.~Oh, 
{\it  Non-existence of solutions for the periodic cubic nonlinear Schr\"odinger equation below $L^2$}, Internat. Math. Res. Not. 2018, no.6, 1656--1729. 

\bibitem{K2019}
N.~Kishimoto, 
{\it Unconditional uniqueness of solutions for nonlinear dispersive equations}, 
arXiv:1911.04349 [math.AP].




\bibitem{Kwak}
C.~Kwak, 
{\it Periodic fourth-order cubic NLS: local well-posedness and non-squeezing property}, 
J. Math. Anal. Appl. 461 (2018), no. 2, 1327--1364.


\bibitem{KOY}
S.~Kwon, T.~Oh,  H.~Yoon,
{\it  Normal form approach to unconditional well-posedness of nonlinear dispersive PDEs on the real line},
 Ann. Fac. Sci. Toulouse Math. 
29 (2020), no. 3, 649--720. 

\bibitem{LSZ} 
G.~Li, K.~Seong, Y.~Zine, 
{\it Global well-posedness of the fractional Schr\"odinger equations on the real line and the circle},
preprint.


\bibitem{MT}
T.~Miyaji, Y.~Tsutsumi, 
{\it Local well-posedness of the NLS equation with third order dispersion in negative Sobolev spaces,} Differential Integral Equations 31 (2018), no. 1-2, 111--132.

\bibitem{MPS}
L.~Molinet, D.~Pilod,  S.~Vento, 
{\it On unconditional well-posedness for the periodic modified Korteweg--de Vries equation,} J. Math. Soc. Japan 71 (2019), no. 1, 147--201. 

\bibitem{NTT}
K.~Nakanishi, H.~Takaoka, Y.~Tsutsumi, 
{\it Local well-posedness in low regularity of the mKdV equation with periodic boundary condition,}
Discrete Contin. Dyn. Syst. 28 (2010), no. 4, 1635--1654.

\bibitem{OST}
T.~Oh, P.~Sosoe, N.~Tzvetkov,
{\it  An optimal regularity result on the quasi-invariant Gaussian measures for the cubic fourth order nonlinear 
	Schr\"odinger equation,} 
J. \'Ec. polytech. Math. 5
(2018), 793--841. 


\bibitem{OS}
T~ Oh, C.~Sulem, 
{\it On the one-dimensional cubic nonlinear Schr\"odinger equation below $L^2$}, Kyoto J.
Math. 52 (2012), no.1, 99--115.

\bibitem{OTT}
T.~Oh, Y.~Tsutsumi, N.~Tzvetkov,
{\it  Quasi-invariant Gaussian measures for the cubic nonlinear Schr\"odinger equation with third order dispersion}, C. R. Math. Acad. Sci. Paris 357 (2019), no. 4, 366--381. 

\bibitem{OTz}
T.~Oh, N.~Tzvetkov,
{\it Quasi-invariant Gaussian measures for the cubic fourth order nonlinear Schr\"odinger equation}, 
Probab. Theory Related Fields 169 (2017), 1121--1168. 



\bibitem{OTz2}
T.~Oh, N.~Tzvetkov,
{\it Quasi-invariant Gaussian measures for the two-dimensional defocusing cubic nonlinear wave
equation}, 
J. Eur. Math. Soc. 22 (2020), no. 6, 1785--1826.



\bibitem{OTzW}
T.~Oh, N.~Tzvetkov, Y.~Wang, 
{\it  Solving the $\text{4NLS}$ with white noise initial data}, 
 Forum Math. Sigma. 8 (2020), e48, 63 pp. 



\bibitem{OW1}
T.~Oh, Y.~Wang,
{\it  Global well-posedness of the periodic cubic fourth order NLS in negative Sobolev spaces},
Forum Math. Sigma 6 (2018), e5, 80 pp. 


\bibitem{OW2}
T.~Oh, Y.~Wang, 
{\it Normal form approach to the one-dimensional periodic cubic nonlinear Schr\"odinger equation in almost critical Fourier-Lebesgue spaces},  J. Anal. Math.
(2021). https://doi.org/10.1007/s11854-021-0168-1 

\bibitem{PTV} 
F.~Planchon, N.~Tzvetkov, N.~Visciglia, 
{\it  Transport of gaussian measures by the 
flow of the nonlinear Schr\"odinger equation}, 
 Math. Ann. 378 (2020), no. 1-2, 389--423.



\bibitem{Ramer}
R.~Ramer, {\it On nonlinear transformations of Gaussian measures}, J. Funct. Anal. 15 (1974), 166--187.


\bibitem{STX}
P.~Sosoe, W.~Trenberth, T.~Xian,
{\it Quasi-invariance of fractional Gaussian fields by the nonlinear wave equation with polynomial nonlinearity}, 
 Differential Integral Equations 33 (2020), no. 7-8, 393--430.


\bibitem{TT}
H.~Takaoka, Y.~Tsutsumi, 
{\it Well-posedness of the Cauchy problem for the modified KdV equation with periodic boundary condition}, Int. Math. Res. Not. 2004, no. 56, 3009--3040.


\bibitem{Tz}
N.~Tzvetkov, {\it Quasi-invariant Gaussian measures for one dimensional Hamiltonian PDE's,}
Forum Math. Sigma 3 (2015), e28, 35 pp. 


\end{thebibliography}
\end{document}